\documentclass{amsart}
\usepackage{graphicx, color}
\usepackage{amscd}
\usepackage{amsmath,empheq}
\usepackage{amsfonts}
\usepackage{amssymb}

\newtheorem{theorem}{Theorem}

\newtheorem{prop}[theorem]{Proposition}
\newtheorem{remark}{Remark}

\newtheorem{claim}{Claim}

\newenvironment{proof-sketch}{\noindent{\bf Sketch of Proof}\hspace*{1em}}{\qed\bigskip}

\everymath{\displaystyle}

\newcommand{\RR}{\mathbb R}
\newcommand{\NN}{\mathbb N}

\newcommand{\ZZ}{\mathbb Z}

\renewcommand{\leq}{\leqslant}

\renewcommand{\geq}{\geqslant} 

\begin{document}

\title[Robin problems]{Robin problems with a general potential and a superlinear reaction}

\author[N.S. Papageorgiou]{N.S. Papageorgiou}
\address[N.S. Papageorgiou]{National Technical University, Department of Mathematics,
				Zografou Campus, 15780 Athens, Greece 
				\& Institute of Mathematics, Physics and Mechanics, 1000 Ljubljana, Slovenia}
\email{\tt npapg@math.ntua.gr}
\author[V.D. R\u{a}dulescu]{V.D. R\u{a}dulescu}
\address[V.D. R\u{a}dulescu]{Institute of Mathematics, Physics and Mechanics, 1000 Ljubljana, Slovenia \&  Faculty of Applied Mathematics, AGH University of Science and Technology, 30-059 Krak\'ow, Poland}
\email{\tt vicentiu.radulescu@imfm.si}
\author[D.D. Repov\v{s}]{D.D. Repov\v{s}}
\address[D.D. Repov\v{s}]{Faculty of Education and Faculty of Mathematics and Physics,
University of Ljubljana, \& Institute of Mathematics, Physics and Mechanics, 1000 Ljubljana, Slovenia}\email{dusan.repovs@guest.arnes.si}


\begin{abstract}
We consider semilinear Robin problems driven by the negative Laplacian plus an indefinite potential and with a superlinear reaction term which need not satisfy the Ambrosetti-Rabinowitz condition. We prove existence and multiplicity theorems (producing also an infinity of smooth solutions) using variational tools, truncation and perturbation techniques and Morse theory (critical groups).
\end{abstract}

\maketitle


\section{Introduction}

In this paper we study the following semilinear elliptic problem with Robin boundary condition:
\begin{equation}\label{eq1}
	\left\{\begin{array}{ll}
		-\Delta u(z)+\xi(z)u(z)=f(z,u(z))\ \mbox{in}\ \Omega,\\
		\frac{\partial u}{\partial n}+\beta(z)u=0\ \mbox{on}\ \partial\Omega.
	\end{array}\right\}
\end{equation}

In this problem $\Omega\subseteq\RR^N$ is a bounded domain with a $C^2$-boundary $\partial\Omega$. The potential function $\xi\in L^s(\Omega)$ with $s>N$ is in general sign-changing. So, the linear part of (\ref{eq1}) is indefinite. The reaction term $f(z,x)$ is a Carath\'eodory function (that is, for all $x\in\RR,\ z\mapsto f(z,x)$ is measurable and for almost all $z\in\Omega,\ x\mapsto f(z,x)$ is continuous), which exhibits superlinear growth near $\pm\infty$. However, $f(z,\cdot)$ does not satisfy the (usual in such cases) Ambrosetti-Rabinowitz condition (AR-condition, for short). Instead, we employ a more general condition which incorporates in our framework superlinear functions with ``slower" growth near $\pm\infty$, which fail to satisfy the AR-condition. Another nonstandard feature of our work is that $f(z,\cdot)$ does not have subcritical polynomial growth. In our case, the growth of $f(z.\cdot)$ is almost critical in the sense that $\lim\limits_{x\rightarrow\pm\infty}\frac{f(z,x)}{|x|^{2^*-2}x}=0$ uniformly for almost all $z\in\Omega$, with $2^*$ being the Sobolev critical exponent for 2, defined by
\begin{eqnarray*}
	2^*=\left\{\begin{array}{ll}
		\frac{2N}{N-2} 	& \mbox{if}\ N\geq 3\\
		+\infty			& \mbox{if}\ N=1,2.
	\end{array}\right.
\end{eqnarray*}

In the boundary condition, $\frac{\partial u}{\partial n}$ denotes the normal derivative of $u\in H^1(\Omega)$ defined by extension of the continuous linear map
$$C^1(\overline{\Omega})\ni u\mapsto\frac{\partial u}{\partial n}=(Du,n)_{\RR^N},$$
with $n(\cdot)$ being the outward unit normal on $\partial\Omega$. The boundary coefficient is $\beta\in W^{1,\infty}(\partial\Omega)$ and we assume that $\beta(z)\geq0$ for all $z\in\partial\Omega$. When $\beta\equiv0$, we have the usual Neumann problem.

Our aim in this paper is to prove existence and multiplicity results within this general analytical framework. Recently, there have been such results primarily for Dirichlet problems. We mention the works of
Lan and Tang \cite{11} (with $\xi\equiv0$), Li and Wang \cite{12}, Miyagaki and Souto \cite{14} (with $\xi\equiv0$), Papageorgiou and Papalini \cite{18}, Qin, Tang and Zhang \cite{24}, Wu and An \cite{29}, Zhang-Liu \cite{30}. For Neumann and Robin problems, we mention the works of D'Agui, Marano and Papageorgiou \cite{3}, Papageorgiou and R\u adulescu {\cite{20, 21,pr17}}, {Papageorgiou, R\u adulescu and Repov\v{s} \cite{prr17}}, Papageorgiou and Smyrlis \cite{23}, {Pucci {\it et al.} \cite{apv13, cpv12}}, Shi and Li \cite{26}. Superlinear problems were treated by Lan and Tang \cite{11}, Li and Wang \cite{12}, Miyagaki and Souto \cite{14}, who proved only existence results. The superlinear case was not studied in the context of Neumann and Robin problems.

Our approach uses variational methods based on the critical point theory, together with suitable truncation and perturbation techniques and Morse theory (critical groups).

\section{Mathematical Background}\label{sec2}

Let $X$ be a Banach space and $X^*$ its topological dual. By $\left\langle \cdot,\cdot\right\rangle$ we denote the duality brackets for the pair $(X^*,X)$. Given $\varphi\in C^1(X,\RR)$, we say that $\varphi$ satisfies the ``Cerami condition" (the ``C-condition" for short), if the following property holds:

 ``Every sequence $\{u_n\}_{n\geq1}\subseteq X$ such that $\{\varphi(u_n)\}_{n\geq1}\subseteq\RR$ is bounded and
$$(1+||u_n||)\varphi'(u_n)\rightarrow0\ \mbox{in}\ X^*\ \mbox{as}\ n\rightarrow\infty,$$
\ \ \, admits a strongly convergent subsequence".

This is a compactness-type condition on $\varphi$, which compensates for the fact that the ambient space $X$ is in general not locally compact. It leads to a deformation theorem from which one can derive the minimax theory of the critical values of $\varphi$. A fundamental result of this theory is the so-called ``mountain pass theorem", which we state here in a slightly more general form (see Gasinski and Papageorgiou \cite[p. 648]{6}). {We also point out that Theorem \ref{th1} is a direct consequence  of  Ekeland \cite[Corollaries 4 and 9]{eke}. }

\begin{theorem}\label{th1}
	Let $X$ be a Banach space. Assume that $\varphi\in C^1(X,\RR)$ satisfies the C-condition and for some $u_0,u_1\in X$ with $||u_1-u_0||>r>0$ we have
	$$\max\{\varphi(u_0),\varphi(u_1)\}<\inf[\varphi(u):||u-u_0||=r]=m_r$$
	and $c=\inf\limits_{\gamma\in \Gamma}\max\limits_{0\leq t\leq 1}\varphi(\gamma(t))\ \mbox{with}\ \Gamma=\{\gamma\in C([0,1],X):\gamma(0)=u_0,\gamma(1)=u_1\}$. Then $c\geq m_r$ and $c$ is a critical value of $\varphi$ (that is, there exists $u_0\in X$ such that $\varphi'(u_0)=0$ and $\varphi(u_0)=c$).
\end{theorem}

It is well known that when the functional $\varphi$ has symmetry properties, then we can have an infinity of critical points. In this direction, we mention two such results which we will use in the sequel. The first is the so-called ``symmetric mountain pass theorem" due to Rabinowitz \cite[Theorem 9.12, p. 55]{25} (see also Gasinski and Papageorgiou \cite[Corollary 5.4.35, p. 688]{6}).
\begin{theorem}\label{th2}
	Let $X$ be an infinite dimensional Banach space such that $X=Y\oplus V$ with $Y$ finite dimensional. Assume that $\varphi\in C^1(X,\RR)$ satisfies the C-condition and that
	\begin{itemize}
		\item [(i)] there exist $\vartheta,\rho>0$ such that $\varphi|_{\partial B_{\rho}\cap V}\geq\vartheta>0$ (here $\partial B_\rho=\{x\in X:||x||=\rho\}$);
		\item [(ii)] for every finite dimensional subspace $E\subseteq X$, we can find $ R=R(E)$ such that $\varphi|_{X\backslash B_R}\leq 0$ (here $B_R=\{u\in X:||u||<R\}$).
	\end{itemize}
	Then $\varphi$ has an unbounded sequence of critical points.
\end{theorem}

The second such abstract multiplicity result that we will need, is a variant of a classical result of Clark \cite{2}, due to Heinz \cite{8} and Kajikiya \cite{10}.
\begin{theorem}\label{th3}
	If $X$ is a Banach space, $\varphi\in C^1(X,\RR)$ satisfies the C-condition, is even and bounded below, $\varphi(0)=0$ and for every $n\in\NN$ there exist an $n$-dimensional subspace $Y_n$ of $X$ and $\rho_n>0$ such that
	$$\sup[\varphi(u):u\in Y_n\cap\partial B_{\rho_n}]<0$$
	then there exists a sequence $\{u_n\}_{n\geq1}\subseteq X$ of critical points of $\varphi$ such that $u_n\rightarrow 0$ in $X$.
\end{theorem}

In the analysis of problem (\ref{eq1}), we will use the following three spaces:
\begin{itemize}
	\item [$\bullet$] the Sobolev space $H^1(\Omega)$;
	\item [$\bullet$] the Banach space $C^1(\overline{\Omega})$;
	\item [$\bullet$] the ``boundary" Lebesgue spaces $L^q(\partial\Omega)$ with $1\leq q\leq \infty$.
\end{itemize}

The Sobolev space $H^1(\Omega)$ is a Hilbert space with inner product given by
$$(u,h)_{H^1}=\int_\Omega uhdz+\int_\Omega(Du,Dh)_{\RR^N}dz\ \mbox{for all}\ u,h\in H^1(\Omega).$$

By $||\cdot||$ we denote the corresponding norm defined by
$$||u||=\left[ ||u||^2_2 + ||Du||^2_2 \right]^{^1/_2}\ \mbox{for all}\ u\in H^1(\Omega).$$

The Banach space $C^1(\overline{\Omega})$ is an ordered Banach space with positive (order) cone given by
$$C_+=\{u\in C^1(\overline{\Omega}):u(z)\geq0\ \mbox{for all}\ z\in\overline{\Omega}\}.$$

This cone has a nonempty interior given by
$$D_+=\{u\in C_+:u(z)>0\ \mbox{for all}\ z\in\overline{\Omega}\}.$$

On $\partial\Omega$ we consider the $(N-1)$-dimensional Hausdorff (surface) measure $\sigma(\cdot)$. Using this measure, we can define in the usual way the ``boundary" Lebesgue space $L^q(\partial\Omega),1\leq q\leq\infty$. From the theory of Sobolev spaces, we know that there exists a unique continuous linear map $\gamma_0:H^1(\Omega)\rightarrow L^2(\partial\Omega)$, known as the ``trace map", such that
$$\gamma_0(u)=u|_{\partial\Omega}\ \mbox{for all}\ u\in H^1(\Omega)\cap C(\overline{\Omega}).$$

So, the trace map extends the notion of boundary values to every Sobolev function. We know that the map $\gamma_0$ is compact into $L^q(\partial\Omega)$ for all $q\in\left[1,\frac{2(N-1)}{N-2}\right)$ if $N\geq3$ and into $L^q(\partial\Omega)$ for all $q\geq1$ if $N=1,2$. Moreover, we have
$${\rm ker}\,\gamma_0=H^1_0(\Omega)\ \mbox{and}\ {\rm im}\,\gamma_0=H^{\frac{1}{2},2}(\partial\Omega).$$

In the sequel, for the sake of notational simplicity, we drop the use of the trace map. All restrictions of Sobolev functions on $\partial\Omega$ are understood in the sense of traces.

We will need some facts about the spectrum of the differential operator $u\mapsto -\Delta u+\xi(z)u$ with Robin boundary condition. So, we consider the following linear eigenvalue problem:
\begin{equation}\label{eq2}
	\left\{\begin{array}{l}
		-\Delta u(z)+\xi(z)u(z)=\hat{\lambda}u(z)\ \mbox{in}\ \Omega,\\
		\frac{\partial u}{\partial n}+\beta(z)u=0\ \mbox{on}\ \partial\Omega.
	\end{array}\right\}
\end{equation}

We assume that
$$\xi\in L^s(\Omega)\ \mbox{with}\ s>N\ \mbox{and}\ \beta\in W^{1,\infty}(\partial\Omega)\ \mbox{with}\ \beta(z)\geq0\ \mbox{for all}\ z\in\partial\Omega.$$

Let $\gamma:H^1(\Omega)\rightarrow\RR$ be the $C^1$-functional defined by
$$\gamma(u)=||Du||^2_2+\int_\Omega\xi(z)u^2 dz+\int_{\partial\Omega}\beta(z)u^2d\sigma\ \mbox{for all}\ u\in H^1(\Omega).$$

From D'Agui, Marano and Papageorgiou \cite{3}, we know that we can find $\mu>0$ such that
\begin{equation}\label{eq3}
	\gamma(u)+\mu||u||^2_2\geq c_0||u||^2\ \mbox{for all}\ u\in H^1(\Omega),\ \mbox{some}\ c_0>0.
\end{equation}

Using (\ref{eq3}) and the spectral theorem for compact self-adjoint operators on a Hilbert space, we show that the spectrum $\hat{\sigma}$(\ref{eq2}) of (\ref{eq2}) consists of a sequence $\{\hat{\lambda}_k\}_{k\in\NN}$ of eigenvalues such that $\hat{\lambda}_k\rightarrow+\infty$. By $E(\hat{\lambda}_k)$ $(k\in \NN)$ we denote the corresponding eigenspace. These items have the following properties:
\begin{itemize}
	\item [$\bullet$] $\hat{\lambda}_1$ is simple (that is, ${\rm dim}\,E(\hat{\lambda}_1)=1$) and
		\begin{equation}\label{eq4}
			\hat{\lambda}_1=\inf\left[ \frac{\gamma(u)}{||u||^2_2}:u\in H^1(\Omega),u\neq0 \right].
		\end{equation}
	\item [$\bullet$] For every $m\geq2$ we have
		\begin{eqnarray}
			\hat{\lambda}_m & = & \inf\left[ \frac{\gamma(u)}{||u||^2_2}:u\in \overline{\mathop{\oplus}\limits_{k\geq m}E(\hat{\lambda}_k)},u\neq0 \right] \nonumber\\
							& = & \sup\left[ \frac{\gamma(u)}{||u||^2_2}:u\in \underset{\text{k=1}}{\overset{\text{m}}{\oplus}}E(\hat{\lambda}_k), u\neq0 \right].\label{eq5}
		\end{eqnarray}
	\item [$\bullet$] For every $k\in\NN,\ E(\hat{\lambda}_k)$	is finite dimensional, $E(\hat{\lambda}_k)\subseteq C^1(\overline{\Omega})$ and it has the ``unique continuation property" (UCP for short), that is, if $u\in E(\hat{\lambda}_k)$ and vanishes on a set of positive measure, then $u\equiv0$.
\end{itemize}

Note that in (\ref{eq4}) the infimum is realized on $E(\hat{\lambda}_1)$ and in (\ref{eq5}) both the infimum and the supremum, are realized on $E(\hat{\lambda}_m)$. The above properties, imply that the nontrivial elements of $E(\hat{\lambda}_1)$ have constant sign, while the nontrivial elements of $E(\hat{\lambda}_m)$ (for $m\geq2$) are all nodal (that is, sign changing) functions. By $\hat{u}_1$ we denote the $L^2$-normalized (that is, $||\hat{u}_1||_2=1$) positive eigenfunction. We know that $\hat{u}_1\in C_+$ and by the Harnack inequality (see, for example, Motreanu, Motreanu and Papageorgiou \cite[p. 211]{15}), we have $\hat{u}_1(z)>0$ for all $z\in\Omega$. Moreover, assuming that $\xi^+\in L^\infty(\Omega)$ and using the strong maximum principle, we have $\hat{u}_1\in D_+$.

Now let $f_0 :\Omega\times\RR \rightarrow\RR$ be a Carath\'eodory function satisfying
$$|f_0(z,x)|\leq a_0(z)(1+|x|^{2^*-1})\ \mbox{for almost all}\ z\in\Omega,\ \mbox{all}\ x\in\RR\,.$$

We set $F_0(z,x)=\int^x_0 f_0(z,s)ds$ and consider the $C^1$-functional $\varphi_0:H^1(\Omega)\rightarrow\RR$ defined by
$$\varphi_0(u)=\frac{1}{2}\gamma(u)-\int_\Omega F(z,u)dz\ \mbox{for all}\ u\in H^1(\Omega).$$

The next result is a special case of a more general result of Papageorgiou and R\u adulescu \cite{19, 22}.
\begin{prop}\label{prop4}
	Assume that $u_0\in H^1(\Omega)$ is a local $C^1(\overline{\Omega})$-minimizer of $\varphi_0$, that is, there exists $\delta_1>0$ such that
	$$\varphi_0(u_0)\leq\varphi_0(u_0+h)\ \mbox{for all}\ h\in C^1(\overline{\Omega})\ \mbox{with}\ ||h||_{C^1(\overline{\Omega})}\leq\delta_1.$$
	Then $u_0\in C^1(\overline{\Omega})$ and $u_0$ is also a local $H^1(\Omega)$-minimizer of $\varphi_0$, that is, there exists $\delta_2>0$ such that
	$$\varphi_0(u_0)\leq\varphi_0(u_0+h)\ \mbox{for all}\ h\in H^1(\Omega)\ \mbox{with}\ ||h||\leq\delta_2.$$
\end{prop}

Next let us recall a few basic definitions and facts from Morse theory, which we will need in the sequel. So, let $X$ be a Banach space, $\varphi\in C^1(X,\RR)$ and $c\in\RR$. We introduce the following sets:
$$\varphi^c=\{ u\in X:\varphi(u)\leq c \},\ K_\varphi=\{ u\in X:\varphi'(u)=0 \},\ K^c_\varphi=\{ u\in K_\varphi:\varphi(u)=c \}.$$

Let $(Y_1,Y_2)$ be a topological pair such that $Y_2\subseteq Y_1\subseteq X$. For $k\in\NN_0$, let $H_k(Y_1,Y_2)$ denote the $k$th relative singular homology group for the pair $(Y_1,Y_2)$ with integer coefficients (for $k\in-\NN$, we have $H_k(Y_1,Y_2)=0$). Given $u_0\in K^c_\varphi$ isolated, the critical groups of $\varphi$ at $u_0$ are defined by
$$C_k(\varphi,u_0)=H_k(\varphi^c\cap U,\varphi^c\cap U\backslash\{u_0\})\ \mbox{for all}\ k\in\NN_0,$$
with $U$ being a neighbourhood of $u_0$ satisfying $\varphi^c\cap K_\varphi\cap U=\{u_0\}$. The excision property of singular homology implies that this definition of critical groups is independent of the choice of the neighbourhood $U$.

Suppose that $\varphi$ satisfies the C-condition and that $\inf\varphi(K_\varphi)>-\infty$. Let $c<\inf\varphi(K_\varphi)$. The critical groups of $\varphi$ at infinity, are defined by
$$C_k(\varphi, \infty)=H_k(X,\varphi^c)\ \mbox{for all}\ k\in\NN_0.$$

This definition is independent of the choice of $c<\inf\varphi(K_\varphi)$. Indeed, let $c<\hat{c}<\inf\varphi(K_\varphi)$. Then from a corollary of the second deformation theorem (see Motreanu, Motreanu and Papageorgiou \cite[Corollary 5.35, p. 115]{15}) we have that
$$\varphi^c\ \mbox{is a strong deformation retract of}\ \varphi^{\hat{c}}.$$

Therefore, we have
$$H_k(X, \varphi^c)=H_k(X,\varphi^{\hat{c}})\ \mbox{for all}\ k\in\NN_0.$$

We assume that $K_\varphi$ is finite and introduce the following quantities:
\begin{eqnarray*}
	M(t,u)		& = & \sum_{k\geqslant0}\ \mbox{rank}\ C_k(\varphi, u)t^k\ \mbox{for all}\ t\in\RR \mbox{, all}\ u\in K_\varphi,\\
	P(t,\infty)	& =	& \sum_{k\geqslant0}\ \mbox{rank}\ C_k(\varphi, \infty)t^k\ \mbox{for all}\ t\in\RR.
\end{eqnarray*}

Then the Morse relation says that
\begin{equation}\label{eq6}
	\sum_{u\in K_\varphi}\ M(t,u) = P(t,\infty)+(1+t)Q(t),
\end{equation}
where $Q(t)=\sum_{k\geqslant0}\hat{\beta}_k t^k$ is a formal series in $t\in\RR$, with nonnegative integer coefficients $\hat{\beta}_k$.

Finally we fix our notation. So, for $x\in\RR$, we set $x^\pm=\max\{\pm x,0\}$. Then for $u\in H^1(\Omega)$ we define $u^\pm(\cdot)=u(\cdot)^\pm$ and we have
$$u=u^+-u^-\ \mbox{,}\quad | u|=u^+-u^-\ \mbox{,}\quad u^\pm\in H^1(\Omega).$$

Given a measurable function $g:\Omega\times\RR\rightarrow\RR$ (for example, a Carath\'eodory function), by $N_g$ we denote the Nemytskii map corresponding to $g$, that is,
$$N_g(u)(\cdot)=g(\cdot,u(\cdot))\ \mbox{for all}\ u\in H^1(\Omega).$$

Evidently, $z\mapsto N_g(u)(z)$ is measurable on $\Omega$. By $|\cdot|_N$ we denote the Lebesgue measure on $\RR^N$. We set
$$m_+=\min\{k\in\NN:\hat{\lambda}_k>0\}\ \mbox{and}\ m_-=\max\{k\in\NN:\hat{\lambda}_k<0\}.$$

Then we have the following orthogonal direct sum decomposition of the Sobolev space $H^1(\Omega)$:
$$H^1(\Omega)=H_-\oplus E(0)\oplus H_+$$
with $H_-=\underset{\text{k=1}}{\overset{\text{m}_-}{\oplus}} E(\hat{\lambda}_k)$, $H_+=\overline{\mathop{\oplus}\limits_{k\geqslant m_+}E(\hat{\lambda}_k)}$. So, every $u\in H^1(\Omega)$ admits a unique sum decomposition
$$u=\overline{u}+u^0+\hat{u}\ \mbox{, with}\ \overline{u}\in H_-\ \mbox{, }\ u^0\in E(0),\ \hat{u}\in H_+$$
\begin{itemize}
	\item [] If $0\notin\hat{\sigma}(2)=\{\hat{\lambda}_k\}_{k\in\NN}$, then $E(0)={0}$ and $m_-=m_+-1$.
	\item [] If $\xi\geqslant0$ and $\xi\not\equiv 0$ or $\beta\not\equiv0$, then $\hat{\lambda}_1>0$ and so $m_+=1$ and $m_-=0$.
	\item [] If $\xi\equiv0$, $\beta\equiv0$ (Neumann problem with zero potential), then
	$$m_+=2,\ m_-=0,\ E(0)=\RR.$$
	\item [] If $u,v\in H^1(\Omega)$, and $v\leq u$, then by $[v,u]$ we denote the order interval defined by
	$$[v,u]=\{y\in H^1(\Omega):\ v(z)\leq y(z)\leq u(z)\ \mbox{for almost all}\ z\in\Omega\}.$$
\end{itemize}

\section{Existence Theorems}

In this section we prove two existence theorems. The two existence results differ on the geometry near the origin of the energy (Euler) functional.

For the first existence theorem, we assume that $f(z,\cdot)$ is strictly sublinear near the origin. More precisely, our hypotheses on the data of problem (\ref{eq1}) are the following:
\begin{itemize}
	\item [$H(\xi):$] $\xi\in L^s(\Omega)$ with $s>\NN$.
	\item [$H(\beta):$] $\beta\in W^{1,\infty}(\partial\Omega)$ with $\beta(z)\geqslant0$ for all $z\in\partial\Omega$.
	\begin{remark}
		When $\beta\equiv0$, we have the usual Neumann problem.
	\end{remark}
	\item [$H(f)_1:$] $f:\Omega\times\RR\rightarrow\RR$ is a Carath\'eodory function such that
	\begin{itemize}
		\item [$(i)$] for every $\rho>0$, there exists $a_\rho\in L^\infty(\Omega)_+$ such that
			$$| f(z,x)|\leq a_\rho(z)\ \mbox{for almost all}\ z\in\Omega\ \mbox{, all}\ | x|\leq\rho$$
			$$\mbox{and}\ \lim_{x\rightarrow\pm\infty}\frac{f(z,x)}{(x)^{2^*-2}x}=0\ \mbox{uniformly for almost all}\ z\in\Omega;$$
		\item [$(ii)$] if $F(z,x)=\int_0^xf(z,s)ds$ and $\tau(z,x)=f(z,x)x-2F(z,x)$, then\\ $\lim_{x\rightarrow\pm\infty}\frac{F(z,x)}{x^2}=+\infty$ uniformly for almost all $z\in\Omega$ and there exists $e\in L^1(\Omega)$ such that
			$$\tau(z,x)\leq\tau(z,y)+e(z)\ \mbox{for almost all}\ z\in\Omega\ \mbox{, all}\ 0\leq x\leq y\ \mbox{and all}\ y\leq x\leq 0;$$
		\item [$(iii)$] $\lim_{x\rightarrow0}\frac{f(z,x)}{x}=0$ uniformly for almost all $z\in\Omega$ and there exists $\delta>0$ such that
			\begin{eqnarray*}
				&[a]& F(z,x)\leq0\ \mbox{for almost all}\ z\in\Omega \mbox{, all}\ | x|\leq\delta\mbox{,}\\
				\mbox{or}&[b]& F(z,x)\geqslant0\ \mbox{for almost all}\ z\in\Omega \mbox{, all}\ |x|\leq\delta\mbox{.}
			\end{eqnarray*}
	\end{itemize}
\end{itemize}
\begin{remark}
Hypothesis $H(f)_1(i)$ is more general than the usual subcritical polynomial growth which says that
$$| f(z,x)|\leq c_1(1+| x|^{r-1})\ \mbox{for almost all}\ z\in\Omega\ \mbox{, all}\ x\in\RR,$$
with $c_1>0$ and $1\leq r<2^*$. Here the growth of $f(z,\cdot)$ is almost critical and this means we face the difficulty that the embedding of $H^1(\Omega)$ into $L^{2^*}(\Omega)$ is not compact. We overcome this difficulty without use of the concentration-compactness principle. Instead we use Vitali's theorem. Hypothesis $H(f)_1(ii)$ implies that
$$\lim_{x\rightarrow\pm\infty}\frac{f(z,x)}{x}=+\infty\ \mbox{uniformly for almost all}\ z\in\Omega.$$
\end{remark}

Therefore $f(z,\cdot)$ is superlinear near $\pm\infty$. Usually such problems are studied using the so-called Ambrosetti-Rabinowitz condition (the AR-condition for short). We recall that the AR-condition says that there exist $q>2$ and $M>0$ such that
\begin{align}
	0 & <  q\ F(z,x)\leq f(z,x)x\ \mbox{for almost all}\ z\in\Omega\ \mbox{, all}\ |x|\geqslant M\label{eq7a}\tag{7a}\\
	0 & \displaystyle <  {\rm essinf}_\Omega F(\cdot,\pm M).\label{eq7b}\tag{7b}
\end{align}
\stepcounter{equation}
Integrating (\ref{eq7a}) and using (\ref{eq7b}), we obtain the following weaker condition
\begin{equation}\label{eq8}
	c_2 | x|^q\leq F(z,x)\ \mbox{for almost all}\ z\in\Omega\ \mbox{, all}\ | x|\geqslant M.
\end{equation}

From (\ref{eq7a}) and (\ref{eq8}), we see that $f(z,\cdot)$ has at least ($q-1$)-polynomial growth. This restriction removes from consideration superlinear functions with ''slower'' growth near $\pm\infty$. For example, consider a function $f(x)$ which satisfies:
$$f(x)=x\left[\ln| x|+\frac{1}{2}\right]\ \mbox{for all}\ | x|\geqslant M.$$

In this case the primitive is $F(x)=\frac{1}{2}x^2\ln |x|$ for all $| x|\geq M$ and so (\ref{eq3}) fails. In particular, then the AR-condition (see (\ref{eq7a}) and (\ref{eq7b})) does not hold. In contrast $f(\cdot)$ satisfies our hypothesis $H(f)_1(ii)$. This condition is a slightly more general form of a condition used by Li and Yang \cite{12}. It is satisfied, if there exists $M>0$ such that
\begin{eqnarray*}
	x\mapsto\frac{f(z,x)}{x}\ \mbox{is nondecreasing on}\ [M,+\infty),\\
	x\mapsto\frac{f(z,x)}{x}\ \mbox{is nonincreasing on}\ (-\infty,-M].
\end{eqnarray*}

Hypothesis $H(f)_1(iii)$ implies that $f(z,\cdot)$ is sublinear near zero.

\textbf{Examples}: The following functions satisfy hypotheses $H(f)_1$. For the sake of simplicity, we drop the $z$-dependence:
\begin{equation*}
	f_1(x)=x\left[\ln(1+|x|)+\frac{|x|}{1+|x|}\right]\ \mbox{for all}\ x\in\RR
\end{equation*}
and
\begin{equation*}
	f_2(x)=\left\{
		\begin{array}{ll}
			\frac{| x|^{2^*-2}x}{\ln(1+|x|)} \left[1-\frac{1}{2^*}\frac{|x|}{(1+| x|)\ln(1+|x|)}\right]-c &\ \mbox{if}\ x<-1\\
			|x|^{r-2}x &\ \mbox{if}\ -1\leq x\leq 1\\
			\frac{| x|^{2^*-2}x}{\ln(1+| x|)} \left[1-\frac{1}{2^*}\frac{| x|}{(1+|x|)\ln(1+| x|)}\right]+c &\ \mbox{if}\ 1<x
		\end{array}
	\right.
\end{equation*}
with $r>2$ and $c=1-\frac{1}{\ln{2}}\left[1-\frac{1}{2^*}\frac{1}{2\ln{2}}\right]$.

Note that we have
\begin{equation*}
	F_1(x) = \frac{1}{2}x^2\ln{(1+|x|)}\ \mbox{for all}\ x\in\RR
\end{equation*}
\begin{equation*}
	F_2(x) = \frac{1}{2^*}\frac{|x|^{2^*}}{\ln{(1+| x|)}} + c|x|\ \mbox{for all}\ |x| \geqslant1.
\end{equation*}

Observe that $f_1(\cdot)$ although superlinear, fails to satisfy the AR-condition, while $f_2$ has almost critical growth.

Let $\varphi:H^1(\Omega)\rightarrow\RR$ be the energy (Euler) functional for problem (\ref{eq1}) defined by
$$\varphi(u)=\frac{1}{2}\gamma(u)-\int_{\Omega}F(z,u)dz\ \mbox{for all}\ u\in H^1(\Omega).$$

Evidently, $\varphi\in C^1(H^1(\Omega))$. First we show that the functional $\varphi$ satisfies the C-condition.

\begin{prop}\label{prop5}
	If hypotheses $H(\xi),H(\beta),H(f)_1(i),(ii)$ hold, then the functional $\varphi$ satisfies the $C$-condition.
\end{prop}
\begin{proof}
	We consider a sequence $\{u_n\}_{n\geq1}\subseteq H^1(\Omega)$ such that
	\begin{eqnarray}
		|\varphi(u_n)|\leq M_1\ \mbox{for some}\ M_1>0,\ \mbox{all}\ n\in\NN, \label{eq9}\\
		(1+||u_n||)\varphi'(u_n)\rightarrow0\ \mbox{in}\ H^1(\Omega)^*\ \mbox{as}\ n\rightarrow\infty \label{eq10}.
	\end{eqnarray}
	
	From (\ref{eq10}) we have
	\begin{equation}\label{eq11}
		\left| \left\langle A(u_n),h \right\rangle+\int_\Omega\xi(z)u_nhdz+\int_{\partial\Omega}\beta(z)u_nhd\sigma-\int_\Omega f(z,u_n)hdz \right|\leq\frac{\varepsilon_n||h||}{1+||u_n||}
	\end{equation}
	for all $h\in H^1(\Omega)$, with $\varepsilon_n\rightarrow 0^+$.
	
	In (\ref{eq11}) we choose $h=u_n\in H^1(\Omega)$ and obtain
	\begin{equation}\label{eq12}
		-\gamma(u_n)+\int_\Omega f(z,u_n)u_n dz\leq\varepsilon_n\ \mbox{for all}\ n\in\NN.
	\end{equation}
	
	On the other hand, by (\ref{eq9}), we have
	\begin{equation}\label{eq13}
		\gamma(u_n)-\int_\Omega 2F(z,u_n)dz\leq2M_1\ \mbox{for all}\ n\in\NN.
	\end{equation}
	
	We add (\ref{eq12}) and (\ref{eq13}) and obtain
	\begin{equation}\label{eq14}
		\int_\Omega\tau(z,u_n)dz\leq M_2\ \mbox{for some}\ M_2>0,\ \mbox{all}\ n\in\NN.
	\end{equation}
	\begin{claim}
		$\{u_n\}_{n\geq1}\subseteq H^1(\Omega)$ is bounded
	\end{claim}
	
	We argue by contradiction. So, we assume that the Claim is not true. Then by passing to a subsequence if necessary, we have
	\begin{equation}\label{eq15}
		||u_n||\rightarrow+\infty .
	\end{equation}
	
	Let $y_n=\frac{u_n}{||u_n||}$ for all $n\in\NN$. Then $||y_n||=1$ and so we may assume that
	\begin{equation}\label{eq16}
		y_n\xrightarrow{w}y\ \mbox{in}\ H^1(\Omega)\ \mbox{and}\ y_n\rightarrow y\ \mbox{in}\ L^{\frac{2s}{s-1}}(\Omega)\ \mbox{and in}\ L^2(\partial\Omega)
	\end{equation}
	(note that since $s>N$ (see hypothesis $H(\xi)$), we have $\frac{2s}{s-1}<2^*$).
	
	First, we assume that $y\neq0$. Let $\Omega^*=\{z\in\Omega:y(z)\neq0\}$. Then $|\Omega^*|_N>0$ and we have
	$$|u_n(z)|\rightarrow+\infty\ \mbox{for almost all}\ z\in\Omega^*\ \mbox{(see (\ref{eq15}))}.$$
	
	Using hypothesis $H(f)_1(ii)$ we have
	$$\frac{F(z,u_n(z))}{||u_n||^2}=\frac{F(z,u_n(z))}{u_n(z)^2}y_n(z)^2\rightarrow+\infty \ \mbox{for almost all}\ z\in\Omega^*.$$
	
	Using Fatou's lemma we can say that
	\begin{equation}\label{eq17}
		\int_{\Omega^*}\frac{F(z,u_n)}{||u_n||^2}dz\rightarrow+\infty\ \mbox{as}\ n\rightarrow\infty
	\end{equation}
	
	Hypothesis $H(f)(ii)$ implies that we can find $M_3>0$ such that
	\begin{equation}\label{eq18}
		F(z,x)\geq0\ \mbox{for almost all}\ z\in\Omega,\ \mbox{all}\ |x|\geq M_3.
	\end{equation}
	
	From (\ref{eq15}) we see that we may assume that
	\begin{equation}\label{eq19}
		||u_n||\geq1\ \mbox{for all}\ n\in\NN.
	\end{equation}
	
	Then we have
	\begin{eqnarray*}
		\int_\Omega\frac{F(z,u_n)}{||u_n||^2}dz & = & \int_{\Omega^*}\frac{F(z,u_n)}{||u_n||^2}dz+\int_{\Omega\backslash\Omega^*}\frac{F(z,u_n)}{||u_n||^2}dz \\
												& = & \int_{\Omega^*}\frac{F(z,u_n)}{||u_n||^2}dz+\int_{(\Omega\backslash\Omega^*)\cap\{|u_n|\geq M_3\}} \frac{F(z,u_n)}{||u_n||^2}dz \\
												&	& + \int_{(\Omega\backslash\Omega^*)\cap\{|u_n|<M_3\}} \frac{F(z,u_n)}{||u_n||^2}dz\\
												& \geq & \int_{\Omega^*}\frac{F(z,u_n)}{||u_n||^2}dz-c_3\ \mbox{for some}\ c_3>0,\ \mbox{all}\ n\in\NN\\
												&	& \mbox{(see (\ref{eq18}), (\ref{eq19}) and hypothesis $H(f)_1(i)$)}
	\end{eqnarray*}
	\begin{equation}\label{eq20}
		\Rightarrow\int_\Omega\frac{F(z,u_n)}{||u_n||^2}dz\rightarrow+\infty\ \mbox{as}\ n\rightarrow\infty\ \mbox{(see (\ref{eq17}))}.
	\end{equation}
	
	By (\ref{eq9}) we have
	\begin{equation}\label{eq21}
		\int_\Omega\frac{F(z,u_n)}{||u_n||^2}dz\leq\frac{M_1}{||u_n||^2}+\frac{1}{2}\gamma(y_n)\leq M_4\ \mbox{for some}\ M_4>0, \mbox{all}\ n\in\NN
	\end{equation}
	(see (\ref{eq19})).
	
	Comparing (\ref{eq20}) and (\ref{eq21}), we get a contradiction.
	
	Now suppose that $y=0$. Let $\eta>0$ and set $v_n=(2\eta)^{^1/_2}y_n\in H^1(\Omega)$ for all $n\in\NN$. Then from (\ref{eq16}) and since $y=0$, we have
	\begin{equation}\label{eq22}
		v_n\xrightarrow{w}0\ \mbox{in}\ H^1(\Omega)\ \mbox{and}\ v_n\rightarrow0\ \mbox{in}\ L^{\frac{2s}{s-1}}(\Omega)\ \mbox{and in}\ L^2(\partial\Omega).
	\end{equation}
	
	Let $c_4=\sup\limits_{n\geq1}||v_n||^{2^*}_{2^*}$ (see (\ref{eq22})). From hypothesis $H(f)_1(i)$ we see that given $\varepsilon>0$, we can find $c_5=c_5(\varepsilon)>0$ such that
	\begin{equation}\label{eq23}
		|F(z,x)|\leq\frac{\varepsilon}{2c_4}|x|^{2^*}+c_5\ \mbox{for almost all}\ z\in\Omega,\ \mbox{all}\ x\in\RR.
	\end{equation}
	
	Let $E\subseteq\Omega$ be a measurable set such that $|E|_N\leq\frac{\varepsilon}{2c_5}$. Then we have
	\begin{eqnarray*}
		\left| \int_E F(z,v_n)dz \right| & \leq & \int_E |F(z,v_n)|dz \\
										 & \leq & \frac{\varepsilon}{2c_4}||v_n||^{2^*}_{2^*}+c_5|E|_N\ \mbox{(see (\ref{eq23}))} \\
										 & \leq & \varepsilon\ \mbox{for all}\ n\in\NN,
	\end{eqnarray*}
	$$\Rightarrow\{F(\cdot,v_n(\cdot))\}_{n\geq1}\subseteq L^1(\Omega)\ \mbox{is uniformly integrable}.$$
	
	Also note that by (\ref{eq22}) and by passing to a subsequence if necessary, we have
	$$F(z,v_n(z))\rightarrow0\ \mbox{for almost all}\ z\in\Omega\ \mbox{as}\ n\rightarrow+\infty\,.$$
	
	So, invoking Vitali's theorem (the extended dominated convergence theorem) we have
	\begin{equation}\label{eq24}
		\int_\Omega F(z,v_n)dz\rightarrow0\ \mbox{as}\ n\rightarrow\infty\,.
	\end{equation}
	
	From (\ref{eq15}), we see that we can find $n_0\in\NN$ such that
	\begin{equation}\label{eq25}
		0<(2\eta)^{^1/_2}\frac{1}{||u_n||}\leq1\ \mbox{for all}\ n\geq n_0.
	\end{equation}
	
	We choose $t_n\in[0,1]$ such that
	\begin{equation}\label{eq26}
		\varphi(t_n u_n)=\max\left[ \varphi(tu_n):0\leq t\leq1 \right].
	\end{equation}
	
	From (\ref{eq25}) and (\ref{eq26}) we have
	\begin{eqnarray}
		\varphi(t_n u_n) & \geq & \varphi(v_n)\nonumber \\
						 & =		& \eta\left[ \gamma(y_n)+\mu||y_n||^2_2 \right]-\int_\Omega F(z,v_n)dz-\mu\eta||y_n||^2_2\nonumber \\
						 & \geq & \eta\left[ c_0-\mu||y_n||^2_2 \right]-\int_\Omega F(z,v_n)dz\ \mbox{for all}\ n\geq n_0\label{eq27} \\
						 &		& \mbox{(see (\ref{eq3}))}.\nonumber
	\end{eqnarray}
	
	Recall that $y=0$. So, from (\ref{eq16}), (\ref{eq24}) and (\ref{eq27}), we see that we can find $n_1\in\NN,\ n_1\geq n_0$ such that
	$$\varphi(t_n u_n)\geq\frac{1}{2}\eta c_0\ \mbox{for all}\ n\geq n_1.$$
	
	But $\eta>0$ was arbitrary. Therefore it follows that
	\begin{equation}\label{eq28}
		\varphi(t_n u_n)\rightarrow+\infty\ \mbox{as}\ n\rightarrow\infty\,.
	\end{equation}
	
	We have
	\begin{equation}\label{eq29}
		\varphi(0)=0\ \mbox{and}\ \varphi(u_n)\leq M_1\ \mbox{for all}\ n\in\NN\ \mbox{(see (\ref{eq9}))}.
	\end{equation}
	
	From (\ref{eq28}) and (\ref{eq29}) we see that we can find $n_2\in\NN$ such that
	\begin{equation}\label{eq30}
		t_n\in(0,1)\ \mbox{for all}\ n\geq n_2.
	\end{equation}
	
	From (\ref{eq26}) and (\ref{eq30}), we have
	\begin{eqnarray}
		&	& \frac{d}{dt}\varphi(tu_n)|_{t=t_n}=0\ \mbox{for all}\ n\geq n_2, \nonumber \\
		& \Rightarrow & \langle \varphi' (t_n u_n),t_n u_n\rangle =0\ \mbox{for all}\ n\geq n_2\ \mbox{(by the chain rule),} \nonumber \\
		& \Rightarrow & \gamma(t_n u_n)=\int_\Omega f(z,t_n u_n)(t_n u_n)dz\ \mbox{for all}\ n\geq n_2. \label{eq31}
	\end{eqnarray}
	
	From hypothesis $H(f)_1(ii)$ and (\ref{eq30}), we have
	\begin{eqnarray}
		&	& \tau(z,t_n u_n)\leq\tau(z,u_n)+e(z)\ \mbox{for almost all}\ z\in\Omega,\ \mbox{all}\ n\geq n_2, \nonumber \\
		&\Rightarrow & \int_\Omega\tau(z,t_n u_n)dz\leq\int_\Omega\tau(z,u_n)dz + ||e||_1\ \mbox{for all}\ n\geq n_2 \nonumber \\
		&\Rightarrow & \int_\Omega f(z,t_n u_n)(t_n u_n)dz\leq c_6 + \int_\Omega 2F(z,t_n u_n)dz \label{eq32} \\
		&	& \mbox{for some $c_6>0$, all $n\geq n_2$ (see (\ref{eq14}))}.\nonumber
	\end{eqnarray}
	
	We return to (\ref{eq31}) and use (\ref{eq32}). Then
	\begin{equation}\label{eq33}
		2\varphi(t_n u_n)\leq c_6\ \mbox{for all}\ n\geq n_2.
	\end{equation}
	
	Comparing (\ref{eq28}) and (\ref{eq33}) we get a contradiction. This proves the Claim.
	
	On account of the Claim, we may assume that
	\begin{equation}\label{eq34}
		u_n\xrightarrow{w}u\ \mbox{in}\ H^1(\Omega)\ \mbox{and}\ u_n\rightarrow u\ \mbox{in}\ L^{\frac{2s}{s-1}}\ \mbox{and in}\ L^2(\partial\Omega).
	\end{equation}
	
	In (\ref{eq11}) we choose $h=u_n-u\in H^1(\Omega)$, pass to the limit as $n\rightarrow\infty$ and use (\ref{eq34}). Then
	\begin{eqnarray*}
		&	& \lim\limits_{n\rightarrow\infty}\langle A(u_n),u_n-u\rangle=0, \\
		& \Rightarrow & ||Du_n||_2\rightarrow||Du||_2 \\
		& \Rightarrow & u_n\rightarrow u\ \mbox{in}\ H^1(\Omega) \\
		&	& \mbox{(by the Kadec-Klee property, see (\ref{eq34}) and  \cite[p. 901]{6})}.
	\end{eqnarray*}
	
	This proves that $\varphi$ satisfies the $C$-condition.
\end{proof}

We assume that $K_\varphi$ is finite (otherwise we already have an infinity of nontrivial solutions for problem (\ref{eq1})). Then the finiteness of $K_\varphi$ and Proposition \ref{prop5} permit the computation of the critical groups of $\varphi$ at infinity.
\begin{prop}\label{prop6}
	If hypotheses $H(\xi),H(\beta),H(f)_1,(i),(ii)$ hold, then $C_k(\varphi,\infty)=0$ for all $k\in \NN$.
\end{prop}
\begin{proof}
	Hypotheses $H(f)_1 (i),(ii)$ imply that given any $\eta>0$, we can find $c_7=c_7(\eta)>0$ such that
	\begin{equation}\label{eq35}
		F(z,x) \geq\frac{\eta}{2}x^2-c_7\ \mbox{for almost all}\	z\in\Omega,\ \mbox{all}\ x\in\RR.
	\end{equation}
	Let $\partial B_1=\{u\in H^1(\Omega):||u||=1\}$. Then for $u\in \partial B_1$ and $t>0$, we have
	\begin{eqnarray}
		\varphi(tu) & = & \frac{t^2}{2}\gamma(u) - \int_\Omega F(z,tu)dz \nonumber \\
					& \leq & \frac{t^2}{2}\left[\gamma(u)-\eta||u||_2^2\right]+c_7|\Omega|_N\ \mbox{(see (\ref{eq35}))} \nonumber\\
					& \leq & \frac{t^2}{2}\left[c_8-\eta||u||_2^2\right]+c_7|\Omega|_N\ \mbox{for some}\ c_8>0 \label{eq36}
	\end{eqnarray}
	(see hypotheses $H(\xi),H(\beta)$).
	
	Recall that $\eta>0$ is arbitrary. So, we can choose $\eta>\frac{c_8}{||u||_2^2}$. Then it follows from (\ref{eq36}) that
	\begin{equation}\label{eq37}
		\varphi(tu) \rightarrow -\infty\ \mbox{as}\ t\rightarrow+\infty\ (u\in\partial B_1).
	\end{equation}
	
	For $u\in\partial B_1$ and $t>0$, we have
	\begin{eqnarray}
\frac{d}{dt}\varphi(tu)&=& \langle\varphi'(tu),u\rangle\ \mbox{(by the chain rule)} \nonumber\\
			& = & \frac{1}{t}\langle\varphi'(tu),tu\rangle \nonumber\\
			& = & \frac{1}{t}\left[\gamma(tu)-\int_\Omega f(z,tu)(tu)dz \right] \nonumber\\
			& \leq & \frac{1}{t}\left[\gamma(tu)-\int_\Omega2F(z,tu)dz+||e||_1 \right]\ \mbox{see hypothesis}\ H(f)_1(ii) \nonumber\\
			& = & \frac{1}{t}\left[2\varphi(tu)+||e||_1 \right].\label{eq38}
	\end{eqnarray}
	
	From (\ref{eq37}) and (\ref{eq38}) we infer that
	$$\frac{d}{dt}\varphi(tu)< 0\ \mbox{for}\ t>0\ \mbox{big}.$$
	Invoking the implicit function theorem, we can find $\vartheta\in C(\partial B_1)$ such that
	\begin{equation}\label{eq39}
		\vartheta>0\ \mbox{and}\ \varphi(\vartheta(u)u)=\rho_0<-\frac{||e||_1}{2}.
	\end{equation}
	
	We extend $\vartheta$ on $H^1(\Omega)\backslash\{0\}$ by defining
	$$\hat\vartheta(u)=\frac{1}{||u||}\vartheta\left(\frac{u}{||u||}\right)\ \mbox{for all}\ u\in H^1(\Omega)\backslash\{0\}. $$
	
	We have that $\hat\vartheta\in C(H^1(\Omega)\backslash\{0\})$ and $\varphi(\hat\vartheta(u)u)=\rho_0$. Also
	\begin{equation}\label{eq40}
		\varphi(u)=\rho_0\Rightarrow\hat\vartheta(u)=1.
	\end{equation}
	
	Therefore, if we define
	\begin{equation}\label{eq41}
		\vartheta_0(u)=\left\{\begin{array}{ll}
			1 				  & \mbox{if}\ \varphi(u)\leq \rho_0 \\
			\hat\vartheta(u)  & \mbox{if}\ \rho_0<\varphi(u)
		\end{array}
		\right.
	\end{equation}	
	then $\vartheta_0\in C^1(H^1(\Omega)\backslash\{0\})\ \mbox{(see (\ref{eq40}))}$.
	
	Consider the deformation $h:[0,1]\times(H^1(\Omega)\backslash\{0\})\rightarrow H^1(\Omega)\backslash\{0\}$ defined by
	$$h(t,u) = (1-t)u+t\vartheta_0(u)u\ \mbox{for all}\ t\in[0,1]\ \mbox{all}\ u\in H^1(\Omega)\backslash\{0\}.$$
	
	We have
	\begin{eqnarray*}
		&&h(0,\cdot)=id|_{H^1(\Omega)\backslash\{0\}} \\
	 	&&h(1,u)\in\varphi^{\rho_0}\ \mbox{(see (\ref{eq41})\ and (\ref{eq40}))}\\
	 	&&h(t,\cdot)|_{\varphi^{\rho_0}} = id|_{\varphi^{\rho_0}}\ \mbox{(see (\ref{eq41}))}.
	\end{eqnarray*}
	
	These  properties imply that
	\begin{equation}\label{eq42}
	\varphi^{\rho_0}\ \mbox{is a strong deformation retract of}\ H^1(\Omega)\backslash\{0\}	.	
	\end{equation}
	
	Consider the map $r_0:H^1(\Omega)\backslash\{0\}\rightarrow \partial B_1$ defined by
	$$r_0(u)=\frac{u}{||u||}\ \mbox{for all}\ u\in H^1(\Omega)\backslash\{0\}.$$
	We see that
	\begin{eqnarray}
		&&r_0(\cdot)\ \mbox{is continous and}\  r_0|_{\partial B_1}=id|_{\partial B_1}, \nonumber\\
		&&\Rightarrow\partial B_1\ \mbox{is a retract of}\ H^1(\Omega)\backslash\{0\}. \label{eq43}
	\end{eqnarray}
	
	Also, if we consider the deformation $h_0:[0,1]\times(H^1(\Omega)\backslash\{0\})\rightarrow H^1(\Omega)\backslash\{0\}$ defined by
	$$h_0(t,u)=(1-t)u+tr_0(u)\ \mbox{for all}\ t\in[0,1],\ \mbox{all}\ u\in H^1(\Omega)\backslash\{0\}$$
	then we see that
	\begin{equation}\label{eq44}
		H^1(\Omega)\backslash\{0\} \ \mbox{is deformable into}\ \partial B_1.
	\end{equation}
	
	From (\ref{eq43}), (\ref{eq44}) and Theorem 6.5 of Dugundji \cite[p. 325]{4} it follows that
	\begin{equation}\label{eq45}
		\partial B_1\ \mbox{is a deformation retract of}\ H^1(\Omega)\backslash\{0\}.
	\end{equation}
	
	From (\ref{eq42}) and (\ref{eq45}) we infer that
	\begin{eqnarray}
		\varphi^{\rho_0}\ \mbox{and}\ \partial B_1\ \mbox{are homotopy equivalent,} \nonumber\\
		\Rightarrow H_k(H^1(\Omega),\varphi^{\rho_0}) = H_k(H^1(\Omega),\partial B_1)\ \mbox{for all}\ k\in\NN_0\label{eq46}
	\end{eqnarray}
	(see Motreanu, Motreanu and Papageorgiou \cite[p. 143]{15}).
	
	Since $\partial B_1$ is the unit sphere of the infinite dimensional Hilbert space $H^1(\Omega)$ then it is contractible (see Gasinski and Papageorgiou \cite[Problems 4.154 and 4.159]{7}). Therefore
	\begin{equation}\label{eq47}
		H_k(H^1(\Omega),\partial B_1) = 0\ \mbox{for all}\ k\in\NN_0	
	\end{equation}
	(see Motreanu, Motreanu and Papageorgiou \cite[p. 147]{15}).
	
	From (\ref{eq46}) and (\ref{eq47}) it follows that
	\begin{equation}\label{eq48}
		H_k(H^1(\Omega),\varphi^{\rho_0}) = 0\ \mbox{for all}\ k\in\NN_0.	
	\end{equation}
	
	Taking $\rho_0$ more negative if necessary (see (\ref{eq39})), we have
	\begin{eqnarray*}
		&&H_k(H^1(\Omega),\varphi^{\rho_0}) = C_k(\varphi,\infty) \ \mbox{for all}\ k\in\NN_0, \\
		&\Rightarrow & C_k(\varphi,\infty ) = 0\ \mbox{for all}\ k\in\NN_0.
	\end{eqnarray*}
\end{proof}

We also compute the critical groups of $\varphi$ at the origin. Recall that we have the orthogonal direct sum decomposition
$$H^1(\Omega)=H_-\oplus E(0)\oplus H_+$$
with $H_-=\underset{\text{k=1}}{\overset{\text{m\_}}{\oplus}}E(\hat{\lambda}_k),H_+=\overline{\mathop{\oplus}\limits_{k\geq m_+}E(\hat{\lambda}_k)}$ (see Section \ref{sec2}). We set
$$d_-={\rm dim}\,H_-\ \mbox{and}\ d^0_-={\rm dim}\, (H_-\oplus E(0)).\ \mbox{Note that}$$
\begin{itemize}
	\item [$\bullet$] $d_-=0$ if $\hat{\lambda}_k\geq0$ for all $k\in\NN_0$ (that is, $H_-=\{0\}$).
	\item [$\bullet$] $d^0_-=0$ if $\hat{\lambda}_k>0$ for all $k\in\NN_0$ (that is, $H_-\oplus E(0)=\{0\}$).
\end{itemize}
\begin{prop}\label{prop7}
	If hypotheses $H(\xi)$, $H(\beta)$, $H(f)_1$ hold, then $C_{d_{-}}(\varphi, 0)\neq0$ or $C_{d_{-}^0}(\varphi, 0)\neq0$.
\end{prop}
\begin{proof}
	First we assume that hypothesis $H(f)_1(iii)$ [a] holds.
	
	Hypotheses $H(f)_1$ imply that given $\varepsilon>0$, we can find $c_9=c_9(\varepsilon)>0$ such that
	\begin{equation}\label{eq49}
		|F(z,x)|\leq \frac{\varepsilon}{2}x^2+c_9|x|^{2^*}\ \mbox{for almost all}\ z\in\RR\ \mbox{all}\ x\in\RR
	\end{equation}
	(if $N=1,2$, then we replace $2^*$ by $r>2$). For $u\in\ H_-$ we have
	\begin{eqnarray}
		\varphi(u) & = & \frac{1}{2}\gamma(u) - \int_\Omega F(z,u)dz\nonumber \\
				   & \leq & \frac{\hat{\lambda}_{m_-}+\varepsilon}{2}||u||_2^2 + c_9||u||_{2^*}^{2^*}\ \mbox{(see (\ref{eq49}) and recall that}\ u\in H_-).\label{eq50}	
	\end{eqnarray}
	
	Recall that $\hat{\lambda}_{m_-}<0$. So, if we choose $\varepsilon\in(0,-\hat{\lambda}_{m_-})$ and exploit the fact that since $H_-$ is finite dimensional all norms are equivalent, then from (\ref{eq50}) we have
	\begin{equation}\label{eq51}
		\varphi(u) \leq -c_{10}|| u||^2 + c_{11}|| u||^{2^*}\ \mbox{for some}\ c_{10}, c_{11}>0,\ \mbox{all}\ u\in H_-.
	\end{equation}
	
	Since $2<2^*$, we see from (\ref{eq51}) that we can find $\rho_1\in(0,1)$ small such that
	\begin{equation}\label{eq52}
		u\in H_-,\ || u|| \leq \rho_1 \Rightarrow \varphi(u)\leq0.
	\end{equation}
	
	Recall that $E(0)$ is finite dimensional. So, all norms are equivalent and we can find $\rho_0>0$ such that
	\begin{equation}\label{eq53}
		u\in E_0,\ || u|| \leq \rho_0 \Rightarrow |u(z)| \leq \delta/_2\ \mbox{for almost all}\ z\in\Omega\,.
	\end{equation}
	Here, $\delta>0$ is as postulated by hypothesis $H(f)_1(iii)$.
	
	Let $u\in E(0)\oplus H_+$. Then $u$ admits a unique sum decomposition
	$$u = u^0+\hat{u}\ \mbox{with}\ u^0\in E(0),\ \hat{u}\in H_+.$$
	
	Note that
	\begin{equation}\label{eq54}
		||u|| \leq \rho_0 \Rightarrow ||u^0|| \leq \rho_0,
	\end{equation}
	since $u^0$ is the orthogonal projection of $u$ on $E(0)$ and the orthogonal projection operator has operator norm equal to 1.
	
	We define $\Omega_{\delta}=\left\{z\in\Omega : |\hat{u}(z)| \leq \frac{\delta}{2} \right\}.$ Then for $u\in E(0)\oplus H_+$ with $|| u|| \leq \rho_0$, we have
	\begin{eqnarray}
		& & |u(z)| \leq|u^0(z)| + |\hat{u}(z)| \leq \frac{\delta}{2}+\frac{\delta}{2}=\delta\ \mbox{for almost all}\ z\in\Omega_\delta \nonumber \\
		& & \mbox{(see (\ref{eq53}), (\ref{eq54})),} \nonumber \\
		& \Rightarrow & \int_{\Omega_\delta}F(z, u(z))dz \leq 0\ \mbox{(see hypothesis $H(f)_1(iii)[a]$)}. \label{eq55}
	\end{eqnarray}
	
	Also, for $u\in E(0)\oplus H_+$ with $|| u||\leq\rho_0$, we have
	\begin{eqnarray}
		| u(z)| \leq | u^0(z)| + |\hat{u}(z)| \leq 2|\hat{u}(z)|\ \mbox{for almost all}\ z\in\Omega \backslash\Omega_\delta \label{eq56} \\
		\mbox{(see (\ref{eq53}), (\ref{eq54}))}. \nonumber
	\end{eqnarray}
	
	So, for $u\in E(0)\oplus H_+$ with $|| u|| \leq \rho_0$, exploiting the orthogonality of the component spaces, we have
	\begin{eqnarray}
		\varphi(u) & = & \frac{1}{2}\gamma(u) - \int_{\Omega}F(z,u)dz \nonumber \\
		& \geqslant & \frac{1}{2	}\gamma(\hat{u}) - \int_{\Omega\backslash\Omega_\delta} F(z,u)dz\ \mbox{(see (\ref{eq55}) and recall that $u^0\in E(0)$)} \nonumber \\
		& \geqslant & \frac{1}{2	}\gamma(\hat{u}) - \varepsilon|| \hat{u}||_{2}^{2} - c_{12}||\hat{u}||^{2^*}\ \mbox{for some $c_{12}>0$ (see (\ref{eq49}), (\ref{eq56}))} \nonumber \\
		& \geqslant & c_{13}||\hat{u}||^2 - c_{12}||\hat{u}||^{2^*}\ \mbox{for some $c_{13}>0$, choosing $\varepsilon >0$ small} \label{eq57} \\
		& & \mbox{(recall that $\hat{u}\in H_+$)}. \nonumber
	\end{eqnarray}
	
	Since $2<2^*$, choosing $\rho_2\in(0,\rho_0]$ small, from (\ref{eq57}) we have
	\begin{equation}\label{eq58}
		\varphi(u)>0 = \varphi(0)\ \mbox{for all $u\in E(0)\oplus H_+$, $0<|| u|| \leq \rho_2$.}
	\end{equation}
	
	Then (\ref{eq52}) and (\ref{eq58}) imply that
	\begin{equation*}
		\left\{
			\begin{array}{ll}
				\varphi \ \mbox{has a local linking at}\ u=0 \\
				\mbox{(for the decomposition $H_-\oplus[E(0)\oplus H_+]$)} \\
				u=0\ \mbox{is a strict minimizer of}\ \varphi\vert_{E(0)\oplus H_+}
			\end{array}
		\right\}
	\end{equation*}
	
	Then by Motreanu, Motreanu and Papageorgiou \cite[pp. 169, 171]{15}, we have
	$$C_{d_-}(\varphi,0)\neq0.$$
	
	Now assume that hypothesis $H(f)_1(iii)$ [b] holds. We consider the following orthogonal direct sum decomposition
	$$H^1(\Omega)=Z\oplus H_+\ \mbox{with}\ Z=H_-\oplus E(0).$$
	
	Then for $u\in H_+$ we have
	\begin{eqnarray*}
		\varphi(u) & = & \frac{1}{2}\gamma(u) - \int_\Omega F(z,u)dz \\
		& \geqslant & \frac{1}{2}\left[\gamma(u) - \varepsilon ||u||^2_2\right]-c_{14}||u||^{2^*}\ \mbox{for some $c_{14}>0$} \\
		& & \mbox{(see (\ref{eq49}))}.
	\end{eqnarray*}
	
	Choosing $\varepsilon>0$ small, we have
	\begin{eqnarray*}
		\varphi(u)\geqslant c_{15} || u||^2 - c_{14}|| u||^{2^*}\ \mbox{for some $c_{15}>0$} \\
		\mbox{(recall that $u\in H_+$).}
	\end{eqnarray*}
	
	Since $2<2^*$, we can find $\rho_1\in(0,1)$ small such that
	\begin{equation}\label{eq59}
		u\in H_+,\ 0<|| u|| \leq \rho_1 \Rightarrow \varphi(0) = 0<\varphi(u).
	\end{equation}
	
	Now suppose that $u\in Z=H_-\oplus E(0)$. The space Z is finite dimensional and so all norms are equivalent. Therefore we can find $\rho_2>0$ such that
	\begin{equation}\label{eq60}
		u\in Z,\ || u|| \leq \rho_2 \Rightarrow | u(z)| \leq \delta\ \mbox{for almost all}\ z\in\Omega\,.
	\end{equation}
	Here, $\delta>0$ is as postulated in hypothesis $H(f)(iii)$. Every $u\in Z$ can be written in a unique way as
	$$u=\overline{u}+u^0\ \mbox{with}\ \overline{u}\in H_-,\ u^0\in E(0).$$
	
	Exploiting the orthogonality of the component spaces, for $u\in Z$ with $|| u|| \leq \rho_2$, we have
	\begin{eqnarray*}
		\varphi(u) & = & \frac{1}{2}\gamma(u) - \int_\Omega F(z,u)dz \\
				   & = & \frac{1}{2}\gamma(\overline{u}) - \int_\Omega F(z,u)dz\ \mbox{(since $u^0\in E(0)$)} \\
				   & \leq & \frac{\hat{\lambda}_{m_-}}{2}|| u||_{2}^{2}\ \mbox{(see (\ref{eq60}) and use hypothesis $H(f)_1(iii)$ [b]).}
	\end{eqnarray*}
	
	Since $\hat{\lambda}_{m_-}<0$, it follows that
	\begin{equation}\label{eq61}
		\varphi(u)\leq0\ \mbox{for all}\ u\in Z=\overline{H}\oplus E(0)\ \mbox{with}\ || u|| \leq \rho_2.
	\end{equation}
	
	Then (\ref{eq59}) and (\ref{eq61}) imply that
	\begin{eqnarray*}
		\left \{
			\begin{array}{l}
				\varphi\ \mbox{has a local linking at}\ u=0 \\
				\mbox{(now for the decomposition $Z\oplus H_+$),} \\
				0\ \mbox{is as strict local minimizer of $\varphi| H_+$}
			\end{array}
		\right \}
	\end{eqnarray*}
	 As before, by Motreanu, Motreanu and Papageorgiou \cite[pp. 169, 171]{15}, we have
	 \begin{equation*}
	 	C_{d_-^0}(\varphi,0)\neq0\ \mbox{(recall $d_-^0={\rm dim}\, Z$).}
	 \end{equation*}
\end{proof}

Now we are ready for our first existence theorem.
\begin{theorem}\label{th8}
	If hypotheses $H(\xi),H(\beta),H(f)_1$ hold, then problem \eqref{eq1} admits a nontrivial solution $\tilde{u}\in C^1(\overline{\Omega})$.
\end{theorem}
\begin{proof}
	From Proposition \ref{prop6}, we have that
	\begin{equation}\label{eq62}
		C_k(\varphi,\infty)=0\ \mbox{for all}\ k\in\NN_0.
	\end{equation}
	
	Also, Proposition \ref{prop7} says that
	\begin{equation}\label{eq63}
		C_{d_-}(\varphi,0)\neq0\ \mbox{or}\ C_{d_-^0}(\varphi,0)\neq0.	
	\end{equation}
	
	Then (\ref{eq62}), (\ref{eq63}) and Corollary 6.92 of Motreanu, Motreanu and Papageorgiou \cite[p. 173]{15} imply that we can find $\tilde{u}\in K_\varphi\backslash\{0\}$. From Papageorgiou and R\u adulescu \cite{19}, we have
	\begin{equation}\label{eq64}
		\left\{\begin{array}{l}
			-\Delta\tilde{u}(z)+\xi(z)\tilde{u}(z)=f(z,\tilde{u}(z))\ \mbox{for almost all}\ z\in\Omega \\
			\frac{\partial\tilde{u}}{\partial n}+\beta(z)\tilde{u}=0\ \mbox{on}\ \partial\Omega .
		\end{array}\right\}
	\end{equation}
	
	We define the following functions
	\begin{equation}\label{eq65}
		a(z)=\left\{\begin{array}{ccc}
			0		&	\mbox{if} & |\tilde{u}(z)|\leq1 \\
			\frac{f(z,\tilde{u}(z))}{\tilde{u}(z)} & \mbox{if} & 1<|\tilde{u}(z)|
		\end{array}\right.
		\mbox{and}\ b(z)=\left\{\begin{array}{ccc}
			f(z,\tilde{u}(z)) & \mbox{if} & |\tilde{u}(z)|\leq1 \\
			0		&	\mbox{if} & 1<|\tilde{u}(z)|	
		\end{array}
		\right.
	\end{equation}
	
	Hypotheses $H(f)_1$, imply that given $\varepsilon>0$, we can find $c_{16}=c_{16}(\varepsilon)>0$ such that
	\begin{equation}\label{eq66}
		|f(z,x)|\leq\varepsilon|x|^{2^*-1}+c_{16}|x|\ \mbox{for almost all}\ z\in\Omega,\ \mbox{all}\ x\in\RR.
	\end{equation}
	
	Then from (\ref{eq65}), (\ref{eq66}) and the Sobolev embedding theorem, we have
	$$a\in L^{\frac{N}{2}}(\Omega).$$
	
	Also, it is clear from (\ref{eq64}) and hypothesis $H(f)_1(i)$ that
	$$b\in L^\infty(\Omega).$$
	
	We rewrite (\ref{eq64}) as follows
	\begin{equation*}
		\left\{\begin{array}{l}
			-\Delta\tilde{u}(z)=(a(z)-\xi(z))\tilde{u}(z)+b(z)\ \mbox{for almost all}\ z\in\Omega, \\
			\frac{\partial\tilde{u}}{\partial n}+\beta(z)\tilde{u}=0\ \mbox{on}\ \partial\Omega.
		\end{array}\right\}		
	\end{equation*}
	
	Invoking Lemma 5.1 of Wang \cite{28}, we have
	$$\tilde{u}\in L^\infty(\Omega).$$
	
	Then hypotheses $H(\xi),H(f)_1(i)$ imply that
	$$f(\cdot,\tilde{u}(\cdot))-\xi(\cdot)\tilde{u}(\cdot)\in L^s(\Omega),\quad s>N.$$
	
	Invoking Lemma 5.2 of Wang \cite{28} (the Calderon-Zygmund estimates), we have
	\begin{eqnarray*}
		&		& \tilde{u}\in W^{2,s}(\Omega), \\
		& \Rightarrow & \tilde{u}\in C^{1,\alpha}(\overline{\Omega})\ \mbox{with}\ \alpha=1-\frac{N}{s}>0 \\
		&		& \mbox{(by the Sobolev embedding theorem).}
	\end{eqnarray*}
\end{proof}

In Theorem \ref{th8}, the reaction term $f(z,\cdot)$ is strictly sublinear near zero (see hypothesis $H(f)_1(iii)$). In the next existence theorem, we change the geometry near zero and assume that $f(z,\cdot)$ is linear near zero. In fact we permit double resonance with respect to any nonprincipal spectral interval.

The new hypotheses on the reaction term $f(z,x)$ are the following:

\smallskip
$H(f)_2:f:\Omega\times\RR\rightarrow\RR$ is a Carath\'eodory function such that hypotheses $H(f)_2(i),(ii)$ are the same as the corresponding hypotheses $H(f)_1(i),(ii)$ and\\
$(iii)$ there exist $m\in N,m\geq2$ and $\delta>0$ such that
$$\hat{\lambda}_m x^2\leq f(z,x)x\leq\hat{\lambda}_{m+1}x^2\ \mbox{for almost all}\ z\in\Omega,\ \mbox{all}\ |x|\leq\delta.$$
\begin{remark}
	The behaviour of $f(z,\cdot)$ near $\pm\infty$ remains the same. However, near zero, the growth of $f(z,\cdot)$ has changed. In fact the new condition for $f(z,\cdot)$ near zero implies linear growth. Also, permits resonance with respect to both endpoints of the nonprincipal spectral interval $[\hat{\lambda}_m,\hat{\lambda}_{m+1}]$, $m\geq2$ (double resonance). This means that the computation of the critical groups of $\varphi$ at the origin changes.
\end{remark}
\begin{prop}\label{prop9}
	If hypotheses $H(\xi),H(\beta),H(f)_2$ hold, then $C_k(\varphi,0)=\delta_{k,d_m}\ZZ$ for all $k\in\NN_0$ with $d_m=dim\underset{\text{i=1}}{\overset{\text{m}}{\oplus}}E(\hat{\lambda}_i)$.
\end{prop}

\begin{proof}
	Let $\vartheta\in(\hat{\lambda}_m,\hat{\lambda}_{m+1})$ and consider the $C^2$-functional $\Psi: H^1(\Omega)\rightarrow\RR$ defined by
	$$\Psi(u) = \frac{1}{2}\gamma(u) - \frac{\vartheta}{2}|| u||_2^2\ \mbox{for all}\ u\in H^1(\Omega).$$
	
	In this case we consider the following orthogonal direct sum decomposition of the Hilbert space $H^1(\Omega)$:
	\begin{equation}\label{eq67}
		H^1(\Omega) = \overline{H}_m \oplus \hat{H}_m\ \mbox{with}\ \overline{H}_m = \mathop{\oplus}\limits_{i=1}^m E(\hat{\lambda}_i), \hat{H}_m = \hat{H}_m^\perp = \overline{\mathop{\oplus}\limits_{i\geqslant m+1} E(\hat{\lambda}_{i})}.
	\end{equation}
	
	The choice of $\vartheta$ and (\ref{eq5}) imply that
	$$\Psi|_{\overline{H}_m} \leq 0\ \mbox{and}\ \Psi|_{\hat{H}_m\setminus\{0\}} > 0.$$
	
	Then Proposition 2.3 of Su \cite{27} implies that
	\begin{equation}\label{eq68}
		C_k(\Psi,0) = \delta_{k,d_m}\ZZ\ \mbox{for all}\ k\in\NN_0.
	\end{equation}
	
	Consider the homotopy $h_*:[0,1]\times H^1(\Omega)\rightarrow\RR$ defined by
	$$h_*(t,u) = (1-t)\varphi(u) + t\Psi(u)\ \mbox{for all}\ t\in[0,1],\ \mbox{all}\ u\in H^1(\Omega).$$
	
	As in the proof of Theorem  \ref{th8}, using the regularity theorem of Wang \cite{28}, we have that
	$$K_{h_{*}(t,\cdot)} \subseteq C^1(\overline{\Omega})\ \mbox{for all}\ t\in[0,1].$$
	
	Let $t>0$ and suppose that $u\in C^1(\overline{\Omega})$ satisfies $0<|| u||_{C^1(\overline{\Omega})} \leq \delta$, with $\delta>0$ as postulated by hypothesis $H(f)_2(iii)$. Then
	\begin{equation}\label{eq69}
		\langle(h_*)'_u(t,u), v\rangle = (1-t)\langle\varphi'(u),v\rangle + t\langle \Psi'(u),v\rangle \ \mbox{for all}\ v\in H^1(\Omega).
	\end{equation}
	
	Using the orthogonal direct sum decomposition (\ref{eq67}), we can write $u$ in a unique way as
	$$u = \overline{u} + \hat{u}\ \mbox{with}\ \overline{u}\in\overline{H}_m,\ \hat{u}\in\hat{H}_m.$$
	
	Exploiting the orthogonality of the component spaces, we have
	$$\langle\varphi'(u),\hat{u}-\overline{u}\rangle = \gamma(\hat{u}) - \gamma(\overline{u}) - \int_\Omega f(z,u)(\hat{u}-\overline{u})dz.$$
	
	Recall the choice of $u\in C^1(\overline{\Omega})$ and use hypothesis $H(f)_2(iii)$. Then
	$$f(z,u(z))(\hat{u}-\overline{u})(z) \leq \hat{\lambda}_{m+1}\hat{u}(z)^2 - \hat{\lambda}_m\overline{u}(z)\ \mbox{for almost all}\ z\in\Omega.$$
	
	Therefore
	\begin{equation}\label{eq70}
		\langle\varphi'(u),\hat{u}-\overline{u}\rangle \geqslant \gamma(\hat{u}) - \hat{\lambda}_{m+1}|| \hat{u}||_2^2 - \left[ \gamma(\overline{u}) - \hat{\lambda}_m|| \overline{u}||_2^2 \right] \geqslant 0
	\end{equation}
	(see (\ref{eq5})).
	
	Also, again via the orthogonality of the component spaces, we have
	\begin{equation}\label{eq71}
		\langle\Psi'(u),\hat{u}-\overline{u}\rangle = \gamma(\hat{u}) - \vartheta|| \hat{u}||_2^2 - \left[ \gamma(\overline{u}) - \vartheta|| \overline{u}||_2^2 \right] \geqslant c_{17}|| u||^2
	\end{equation}
	for some $c_{17}>0$ (recall that $\partial\in(\hat{\lambda}_m,\hat{\lambda}_{m+1})$).
	
	Returning to (\ref{eq69}) and using $v = \hat{u}-\overline{u}\in H^1(\Omega)$ and relations (\ref{eq70}), (\ref{eq71}), we obtain
	$$\langle(h_*)_u'(t,u),\hat{u}-\overline{u}\rangle \geq t c_{17} || u||^2 > 0\ \mbox{for all}\ 0<t\leq 1.$$
	
	For $t=0$, we have $h_*(0,\cdot) = \varphi(\cdot)$ and $0\in K_\varphi$ is isolated (recall that we assumed that $K_\varphi$ is finite or otherwise we already have an infinity of nontrivial solutions). Therefore from the homotopy invariance property of critical groups (see Gasinski and Papageorgiou \cite[Theorem 5.125, p. 836]{7}), we have
	\begin{eqnarray}
		& & C_k(h_*(0,\cdot),0) = C_k(h_*(1,\cdot),0)\ \mbox{for all}\ k\in\NN_0, \nonumber \\
		& \Rightarrow & C_k(\varphi,0) = C_k(\Psi,0)\ \mbox{for all}\ k\in\NN_0, \nonumber \\
		& \Rightarrow & C_k(\varphi,0) = \delta_{k,d_m}\ZZ\ \mbox{for all}\ k\in\NN_0\ \mbox{(see (\ref{eq68})).} \nonumber
	\end{eqnarray}
\end{proof}

Using this Proposition and reasoning as in the proof of Theorem \ref{th8}, we obtain the following existence theorem.
\begin{theorem}\label{th10}
	If hypotheses $H(\xi),H(\beta),H(f_2)$ hold, then problem (\ref{eq1}) admits a nontrivial smooth solution
	$$\tilde{u}\in C^1(\overline{\Omega}).$$
\end{theorem}

\section{Equations with Concave Terms}

In this section we examine what happens when the reaction function exhibits a concave term (that is, a term which is strictly superlinear near zero). So, the geometry is different from both cases considered in Section 3. To deal with this new problem, we introduce a parameter $\lambda>0$ in the concave term and for all $\lambda>0$ small we prove multiplicity results for the equation.

So, now the problem under consideration, is:
\begin{equation}
		\left\{\begin{array}{ll}
			-\Delta u(z)+\xi(z)u(z)=\lambda|u(z)|^{q-2}u(z)+f(z,u(z))&\mbox{in}\ \Omega,\\
			\frac{\partial u}{\partial n}+\beta(z)u=0\ \ \mbox{on}\ \partial\Omega,\ \lambda>0,\ 1<q<2.&
		\end{array}\right\}\tag{$P_{\lambda}$} \label{eqP}
\end{equation}

The hypotheses on the perturbation $f(z,x)$ are the following:

\smallskip
$H(f)_3:$ $f:\Omega\times\RR\rightarrow\RR$ is a Carath\'eodory function such that hypotheses $H(f)_3(i),(ii)$ are the some as the corresponding hypotheses $H(f)_1(i),(ii)$ and
\begin{itemize}
	\item[(iii)] there exist functions $\eta,\eta_0\in L^{\infty}(\Omega)$ such that
	\begin{eqnarray*}
		&&\eta(z)\leq\hat{\lambda}_1\ \mbox{for almost all}\ z\in\Omega,\ \eta\not\equiv\hat{\lambda}_1,\\
		&&\eta_0(z)\leq\lim\limits_{x\rightarrow 0}\frac{f(z,x)}{x}\leq\limsup\limits_{x\rightarrow 0}\frac{f(z,x)}{x}\leq\eta(z)\ \mbox{uniformly for almost all}\ z\in\Omega.
	\end{eqnarray*}
\end{itemize}

For every $\lambda>0$, let $\varphi_{\lambda}:H^1(\Omega)\rightarrow\RR$ be the energy (Euler) functional for problem \eqref{eqP} defined by
$$\varphi_{\lambda}(u)=\frac{1}{2}\gamma(u)-\frac{\lambda}{q}||u||^q_q-\int_{\Omega}F(z,u)dz\ \mbox{for all}\ u\in H^1(\Omega).$$

Evidently, $\varphi_{\lambda}\in C^1(H^1(\Omega))$.

Also, we consider the following truncations-perturbations of the reaction in problem \eqref{eqP}:
\begin{eqnarray}
	&&\hat{f}^+_{\lambda}(z,x)=\left\{\begin{array}{ll}
		0&\mbox{if}\ x\leq 0\\
		\lambda x^{q-1}+f(z,x)+\mu x&\mbox{if}\ 0<x,
	\end{array}\right.\label{eq72}\\
	&&\hat{f}^-_{\lambda}(z,x)=\left\{\begin{array}{ll}
		\lambda |x|^{q-2}x+f(z,x)+\mu x&\mbox{if}\ x\leq 0,\\
		0&\mbox{if}\ 0<x.
	\end{array}\right.\label{eq73}
\end{eqnarray}

Here $\mu>0$ is as in (\ref{eq3}). Both $\hat{f}^+_{\lambda}$ and $\hat{f}^-_{\lambda}$ are Carath\'eodory functions. We set $\hat{F}^{\pm}_{\lambda}(z,x)=\int^x_0\hat{f}^{\pm}_{\lambda}(z,s)ds$ and consider the $C^1$-functionals $\hat{\varphi}^{\pm}_{\lambda}:H^1(\Omega)\rightarrow\RR$ defined by
$$\hat{\varphi}^{\pm}_{\lambda}(u)=\frac{1}{2}\gamma(u)+\frac{\mu}{2}||u||^2_2-\int_{\Omega}\hat{F}^{\pm}_{\lambda}(z,u)dz\ \mbox{for all}\ u\in H^1(\Omega).$$

Since hypotheses $H(f)_3(i),(ii)$ are the same as $H(f)_1(i),(ii)$ and the concave term $\lambda|x|^{q-2}x$ does not affect the behavior of the reaction near $\pm\infty$, from Proposition \ref{prop5}, we have:
\begin{prop}\label{prop11}
	If hypotheses $H(\xi),\,H(\beta),\,H(f)_3$ hold and $\lambda>0$, then the functional $\varphi_{\lambda}$ satisfies the C-condition.
\end{prop}

Using similar arguments, we can also prove the following result:
\begin{prop}\label{prop12}
	If hypotheses $H(\xi),H(\beta),H(f)_3$ hold and $\lambda>0$, then the functionals $\hat{\varphi}^{\pm}_{\lambda}$ satisfy the C-condition.
\end{prop}
\begin{proof}
	As we already indicated, the proof is basically the same as that of Proposition \ref{prop5}. So, we only present the first part of the proof, which differs a little due to the unilateral nature of the functionals $\hat{\varphi}^{\pm}_{\lambda}$ (see (\ref{eq72}) and (\ref{eq73})).
	
	So, let $\{u_n\}_{n\geq 1}\subseteq H^1(\Omega)$ be a sequence such that
	\begin{eqnarray}
		&&|\hat{\varphi}^{+}_{\lambda}(u_n)|\leq M_5\ \mbox{for some}\ M_5>0,\ \mbox{all}\ n\in\NN,\label{eq74}\\
		&&(1+||u_n||)(\hat{\varphi}^+_{\lambda})'(u_n)\rightarrow 0\ \mbox{in}\ H^1(\Omega)^*\ \mbox{as}\ n\rightarrow\infty\,.\label{eq75}
	\end{eqnarray}
	
	From (\ref{eq75}) we have
	\begin{eqnarray}\label{eq76}
		&&
		\left|\left\langle A(u_n),h\right\rangle+\int_{\Omega}(\xi(z)+\mu)u_nhdz
		+	
		\int_{\partial\Omega}\beta(z)u_nhd\sigma
		-
		\int_{\Omega}\hat{f}^+_{\lambda}(z,u_n)hdz\right|
				\\		
		&&\leq
		\frac{\epsilon_n||h||}{1+||u_n||} \ \ \mbox{for all}\ h\in H^1(\Omega),\ \mbox{with}\ \epsilon_n\rightarrow 0^+.\nonumber
	\end{eqnarray}
	
	In (\ref{eq76}) we choose $h=-u^-_n\in H^1(\Omega)$. Then
	\begin{eqnarray}\label{eq77}
		&&\gamma(u^-_n)+\mu||u^-_n||^2_2\leq\epsilon_n\ \mbox{for all}\ n\in\NN\ (\mbox{see\ (\ref{eq72})}),\nonumber\\
		&\Rightarrow&c_0||u^-_n||^2\leq\epsilon_n\ \mbox{for all}\ n\in\NN\ (\mbox{see (\ref{eq3})}),\nonumber\\
		&\Rightarrow&u^-_n\rightarrow 0\ \mbox{in}\ H^1(\Omega).
	\end{eqnarray}
	
	Using (\ref{eq72}) and reasoning as in the proof of Proposition \ref{prop5}, we show that
	\begin{equation}\label{eq78}
		\{u^+_n\}_{n\geq 1}\subseteq H^1(\Omega)\ \mbox{is bounded}.
	\end{equation}
	
	From (\ref{eq77}) and (\ref{eq78}) it follows that
	$$\{u_n\}_{n\geq 1}\subseteq H^1(\Omega)\ \mbox{is bounded}.$$
	
	From this via the Kadec-Klee property, as in the proof of Proposition \ref{prop5}, we establish that
	$$\hat{\varphi}^+_{\lambda}\ \mbox{satisfies the C-condition}.$$
	
	In a similar fashion, using this time (\ref{eq73}), we show that $\hat{\varphi}^-_{\lambda}$ satisfies the C-condition.
\end{proof}

Hypothesis $H(f)_3(ii)$ (the superlinearity condition) implies that the functionals $\hat{\varphi}^{\pm}_{\lambda}$ are unbounded below.
\begin{prop}\label{prop13}
	If hypotheses $H(\xi),H(\beta),H(f)_3$ hold, $\lambda>0$ and $u\in D_+$ then $\hat{\varphi}^{\pm}_{\lambda}(tu)\rightarrow-\infty$ as $t\rightarrow\pm\infty$.
\end{prop}

The next result will help us to verify the mountain pass geometry (see Theorem \ref{th1}) for the functionals $\hat{\varphi}^{\pm}_{\lambda}$ when $\lambda>0$ is small.
\begin{prop}\label{prop14}
	If hypotheses $H(\xi),H(\beta),H(f)_3$ hold, then we can find $\lambda^{\pm}_*>0$ such that for every $\lambda\in(0,\lambda^{\pm}_{*})$ we can find $\rho^{\pm}_{\lambda}>0$ for which we have
	$$\inf[\hat{\varphi}^{\pm}_{\lambda}(u):||u||=\rho^{\pm}_{\lambda}]=\hat{m}^{\pm}_{\lambda}>0.$$
\end{prop}
\begin{proof}
	Hypotheses $H_f(3)(i),(iii)$ imply that given $\epsilon>0$, we can find $c_{\epsilon}>0$ such that
	\begin{equation}\label{eq79}
		F(z,x)\leq\frac{1}{2}(\eta(z)+\epsilon)x^2+c_{\epsilon}|x|^{2^*}\ \mbox{for almost all}\ z\in\Omega,\ \mbox{all}\ x\in\RR.
	\end{equation}
	
	For every $u\in H^1(\Omega)$, we have
	\begin{eqnarray*}
		\hat{\varphi}^{+}_{\lambda}(u)&\geq&\frac{1}{2}\gamma(u)+\frac{\mu}{2}||u^-||^2_2-\frac{1}{2}\int_{\Omega}\eta(z)(u^+)^2dz-\frac{\epsilon}{2}||u^+||^2-c_{18}(\lambda||u||^q+||u||^{2^*})\\
		&&\mbox{for some}\ c_{18}>0\ (\mbox{see (\ref{eq72}), (\ref{eq79})})\\
		&=&\frac{1}{2}\left[\gamma(u^+)-\int_{\Omega}\eta(z)(u^+)^2dz-\epsilon||u^+||^2\right]+\frac{1}{2}[\gamma(u^-)+\mu||u^-||^2_2]-\\
		&&\hspace{0.5cm}c_{18}(\lambda||u||^q+||u||^{2^*})\\
		&\geq&\frac{1}{2}(c_{19}-\epsilon)||u^+||^2+\frac{c_0}{2}||u^-||^2-c_{18}(\lambda||u||^q+||u||^{2^*})
	\end{eqnarray*}
	for some $c_{19}>0$ (see Lemma 4.11 of Mugnai and Papageorgiou \cite{16} and (\ref{eq3})).
	
	So, choosing $\epsilon\in(0,c_{19})$, we have
	\begin{equation}\label{eq80}
		\hat{\varphi}^{\pm}_{\lambda}(u)\geq\left[c_{20}-c_{18}(\lambda||u||^{q-2}+||u||^{2^*-2})\right]||u||^2\ \mbox{for some}\ c_{20}>0.
	\end{equation}
	
	Consider the function
	$$\Im_{\lambda}(t)=\lambda t^{q-2}+t^{2^*-2}\ \mbox{for all}\ t>0,\ \lambda>0.$$
	
	Since $q<2<2^*$, we see that
	$$\Im_{\lambda}(t)\rightarrow+\infty\ \mbox{as}\ t\rightarrow 0^+\ \mbox{and as}\ t\rightarrow+\infty.$$
	
	So, we can find $t_0=t_0(\lambda)\in(0,+\infty)$ such that
	$$\Im_{\lambda}(t_0)=\min[\Im_{\lambda}(t):t>0].$$
	
	We have
	\begin{eqnarray*}
		&&\Im'_{\lambda}(t_0)=0,\\
		&\Rightarrow&\lambda(2-q)t^{q-3}_0=(2^*-2)t_0^{2^*-3},\\
		&\Rightarrow&t_0=t_0(\lambda)=\left[\frac{\lambda(2-q)}{2^*-2}\right]^{\frac{1}{2^*-q}}.
	\end{eqnarray*}
	
	Hence it follows that
	$$\Im_{\lambda}(t_0)\rightarrow 0^+\ \mbox{as}\ \lambda\rightarrow 0^+.$$
	
	So, we can find $\lambda^+_*>0$ such that
	$$\Im_{\lambda}(t_0)<\frac{c_{20}}{c_{18}}\ \mbox{for all}\ \lambda\in(0,\lambda^+_*).$$
	
	Returning to (\ref{eq80}) and using this fact we obtain
	$$\hat{\varphi}^+_{\lambda}(u)\geq\hat{m}^+_{\lambda}>0\ \mbox{for all}\ u\in H^1(\Omega),\ ||u||=\rho^+_{\lambda}=t_0(\lambda).$$
	
	Similarly for the functional $\hat{\varphi}^-_{\lambda}$.
\end{proof}

To produce multiple solutions of constant sign, we need to strengthen the condition on the potential function. The new hypothesis on $\xi(\cdot)$ is:

\smallskip
$H(\xi)':$ $\xi\in L^s(\Omega)$ with $s>N$ and $\xi^+\in L^{\infty}(\Omega)$.
\begin{prop}\label{prop15}
	If hypotheses $H(\xi)',H(\beta),H(f)_3$ hold, then
	\begin{itemize}
		\item[(a)] for every $\lambda\in(0,\lambda^+_*)$ problem \eqref{eqP} has two positive solutions
		$$u_0,\hat{u}\in D_+\ \mbox{with}\ u_0\ \mbox{being a local minimizer of}\ \varphi_{\lambda};$$
		\item[(b)] for every $\lambda\in(0,\lambda^-_*)$ problem \eqref{eqP} has two negative solutions
		$$v_0,\hat{v}\in-D_+\ \mbox{with}\ v_0\ \mbox{being a local minimizer of}\ \varphi_{\lambda};$$
		\item[(c)] for every $\lambda\in(0,\lambda_*)$ (here $\lambda_*=\min\{\lambda^+_*,\lambda^-_*\}$) problem \eqref{eqP} has at least four nontrivial solutions of constant sign
		\begin{eqnarray*}
			&&u_0,\hat{u}\in D_+,\ v_0,\hat{v}\in-D_+,\\
			&&u_0\ \mbox{and}\ v_0\ \mbox{are local minimizers of}\ \varphi_{\lambda}.
		\end{eqnarray*}
	\end{itemize}
\end{prop}
\begin{proof}
	\textit{(a)} Let $\lambda\in(0,\lambda^+_*)$ and let $\rho^+_{\lambda}>0$ be as postulated by Proposition \ref{prop14}. We consider the set
	$$\bar{B}^+_{\lambda}=\{u\in H^1(\Omega):||u||\leq\rho^+_{\lambda}\}.$$
	
	This set is weakly compact. Also, using the Sobolev embedding theorem and the compactness of the trace map, we see that $\hat{\varphi}^+_{\lambda}$ is sequentially weakly lower semicontinuous. So, by the Weierstrass-Tonelli theorem, we can find $u_0\in H^1(\Omega)$ such that
	\begin{equation}\label{eq81}
		\hat{\varphi}^+_{\lambda}(u_0)=\inf[\hat{\varphi}^+_{\lambda}(u):u\in H^1(\Omega)].
	\end{equation}
	
	Since $q<2$, for $t\in(0,1)$ small we have
	\begin{eqnarray*}
		&&\hat{\varphi}^+_{\lambda}(t\hat{u}_1)<0,\\
		&\Rightarrow&\hat{\varphi}^+_{\lambda}(u_0)<0=\hat{\varphi}^+_{\lambda}(0)\ (\mbox{see (\ref{eq81})}),\\
		&\Rightarrow&u_0\neq 0\ \mbox{and}\ ||u_0||<\rho^+_{\lambda}\ (\mbox{see Proposition \ref{prop14}}).
	\end{eqnarray*}
	
	This fact and (\ref{eq81}) imply that
	\begin{eqnarray}\label{eq82}
		&&(\hat{\varphi}^+_{\lambda})'(u_0)=0,\nonumber\\
		&\Rightarrow&\left\langle A(u_0),h\right\rangle+\int_{\Omega}(\xi(z)+\mu)u_0hdz+\int_{\partial\Omega}\beta(z)u_0hd\sigma=\int_{\Omega}\hat{f}^+_{\lambda}(z,u_0)hdz\\
		&&\mbox{for all}\ h\in H^1(\Omega).\nonumber
	\end{eqnarray}
	
	In (\ref{eq82}) we choose $h=-u^-_0\in H^1(\Omega)$. Using (\ref{eq72}) we obtain
	\begin{eqnarray*}
		&&\gamma(u^-_0)+\mu||u^-_0||^2_2=0,\\
		&\Rightarrow&c_0||u^-_0||^2\leq 0\ (\mbox{see (\ref{eq3})}),\\
		&\Rightarrow&u_0\geq 0,\ u_0\neq 0.
	\end{eqnarray*}
	
	Then because of (\ref{eq72}), equation (\ref{eq82}) becomes
	\begin{eqnarray}\label{eq83}
		&&\left\langle A(u_0),h\right\rangle+\int_{\Omega}\xi(z)u_0hdz+\int_{\partial\Omega}\beta(z)u_0hd\sigma=\int_{\Omega}[\lambda u^{q-1}_0+f(z,u_0)]hdz\nonumber\\
		&&\mbox{for all}\ h\in H^1(\Omega),\nonumber\\
		&\Rightarrow&-\Delta u_0(z)+\xi(z)u_0(z)=\lambda u_0(z)^{q-1}+f(z,u_0(z))\ \mbox{for almost all}\ z\in\Omega,\\
		&&\frac{\partial u_0}{\partial n}+\beta(z)u_0=0\ \mbox{on}\ \partial\Omega\ \mbox{(see Papageorgiou and R\u adulescu \cite{19})}.\nonumber
	\end{eqnarray}
	
	As in the proof of Theorem \ref{th8}, from the regularity theory of Wang \cite{28}, we have
	$$u_0\in C_+\backslash\{0\}.$$
	
	Hypotheses $H(f)_3(i),(iii)$, imply that
	$$|f(z,x)|\leq c_{21}|x|\ \mbox{for almost all}\ z\in\Omega,\ \mbox{all}\ |x|\leq ||u_0||_\infty,\ \mbox{some}\ c_{21}>0.$$
	
	Then from (\ref{eq83}) and hypothesis $H(\xi)'$, we have
	\begin{eqnarray}\label{eq84}
		&&\Delta u_0(z)\leq[||\xi^+||_{\infty}+c_{21}]u_0(z)\ \mbox{for almost all}\ z\in\Omega,\nonumber\\
		&\Rightarrow&u_0\in D_+\ (\mbox{by the strong maximum principle}).
	\end{eqnarray}
	
	From (\ref{eq72}) it is clear that
	$$\left.\varphi_{\lambda}\right|_{C_+}=\left.\hat{\varphi}^+_{\lambda}\right|_{C_+}.$$
	
	So, from (\ref{eq84}) it follows that
	\begin{eqnarray*}
		&&u_0\ \mbox{is a local}\ C^1(\overline{\Omega})-\mbox{minimizer of}\ \varphi_{\lambda},\\
		&\Rightarrow&u_0\ \mbox{is a local}\ H^1(\Omega)-\mbox{minimizer of}\ \varphi_{\lambda}\ (\mbox{see Proposition \ref{prop4}}).
	\end{eqnarray*}
	
	Next we will produce a second positive smooth solution.
	
	We know that
	\begin{equation}\label{eq85}
		\hat{\varphi}^+_{\lambda}(u_0)<0<\hat{m}^+_{\lambda}=\inf[\hat{\varphi}^+_{\lambda}(u):||u||=\rho^+_{\lambda}]\ (\mbox{see Proposition \ref{prop14}}).
	\end{equation}
	
	Also, from Propositions \ref{prop12} and \ref{prop13}, we have
	\begin{eqnarray}
		&&\hat{\varphi}^+_{\lambda}\ \mbox{satisfies the C-condition},\label{eq86}\\
		&&\hat{\varphi}^+_{\lambda}(t\hat{u}_1)\rightarrow-\infty\ \mbox{as}\ t\rightarrow+\infty\,.\label{eq87}
	\end{eqnarray}
	
	Then (\ref{eq85}), (\ref{eq86}) and (\ref{eq87}) permit the use of Theorem \ref{th1} (the mountain pass theorem). So, we can find $\hat{u}\in H^1(\Omega)$ such that
	\begin{equation}\label{eq88}
		\hat{u}\in K_{\hat{\varphi}^+_{\lambda}}\ \mbox{and}\ \hat{m}^+_{\lambda}\leq \hat{\varphi}^+_{\lambda}(u).
	\end{equation}
	
	From (\ref{eq85}) and (\ref{eq88}) we have that
	\begin{eqnarray*}
		&&\hat{u}\neq u_0,\ \hat{u}\neq 0,\\
		&\Rightarrow&\hat{u}\in D_+\ \mbox{is a solution of \eqref{eqP} (as before).}
	\end{eqnarray*}
	
	\textit{(b)} Reasoning in a similar fashion, this time using the functional $\hat{\varphi}^-_{\lambda}$ and (\ref{eq73}), we produce two negative smooth solutions
	$$v_0,\hat{v}\in -D,\ v_0\neq\hat{v},$$
	with $v_0$ being a local minimizer of $\varphi_{\lambda}$.
	
	\textit{(c)} Follows from \textsl{(a)} and \textsl{(b)}.
\end{proof}

In fact we can show that problem \eqref{eqP} has extremal constant sign solutions. So, for all $\lambda\in(0,\lambda^+_*)$ there exists a smallest positive solution $u^*_{\lambda}\in D_+$ and for all $\lambda\in(0,\lambda^-_*)$ there is a biggest negative solution $v^*_{\lambda}\in-D_+$.

Let $S^+_{\lambda}$ (respectively $S^-_{\lambda}$) be the set of positive (respectively negative) solutions of problem. From Proposition \ref{prop15} we know that
\begin{eqnarray*}
	&&\emptyset\neq S^+_{\lambda}\subseteq D_+\ \mbox{for all}\ \lambda\in(0,\lambda^+_*),\\
	&&\emptyset\neq S^-_{\lambda}\subseteq D_+\ \mbox{for all}\ \lambda\in(0,\lambda^-_*).
\end{eqnarray*}

Moreover, from Filippakis and Papageorgiou \cite{5} (see Lemmata 4.1 and 4.2), we have that
\begin{itemize}
	\item $S^+_{\lambda}$ is downward directed (that is, if $u_1,u_2\in S^+_{\lambda}$, then there exists $u\in S^+_{\lambda}$ such that $u\leq u_1$ and $u\leq u_2$).
	\item $S^-_{\lambda}$ is upward directed (that is, if $v_1,v_2\in S^-_{\lambda}$, then there exists $v\in S^-_{\lambda}$ such that $v_1\leq v,v_2\leq v$).
\end{itemize}

Note that hypotheses $H(f)_3(i),(iii)$ imply that
\begin{equation}\label{eq89}
	f(z,x)x\geq-c_{22}x^2-c_{23}|x|^{2^*}\ \mbox{for almost all}\ z\in\Omega,\ \mbox{all}\ x\in\RR,\ \mbox{some}\ c_{22},c_{23}>0.
\end{equation}

We may always assume that $c_{22}\geq\mu$ (see (\ref{eq3})). This unilateral growth condition on $f(z,\cdot)$, leads to the following auxiliary Robin problem:
\begin{equation}
	\left\{\begin{array}{ll}
		-\Delta u(z)+\xi(z)u(z)=\lambda|u(z)|^{q-2}u(z)-c_{22}u(z)-c_{23}|u(z)|^{2^*-2}u(z)&\mbox{in}\ \Omega,\\
		\frac{\partial u}{\partial n}+\beta(z)u=0\ \mbox{on}\ \partial\Omega\,.&
	\end{array}\right\}\tag{$Au_{\lambda}$}\label{eqA}
\end{equation}
\begin{prop}\label{prop16}
	If hypotheses $H(\xi)',H(\beta),H(f)_3$ hold and $\lambda>0$, then problem \eqref{eqA} has a unique positive solution
	$$\tilde{u}_{\lambda}\in D_+$$
	and because problem \eqref{eqA} is odd it follows that
	$$\tilde{v}_{\lambda}=-\tilde{u}_{\lambda}\in-D_+$$
	is the unique negative solution of \eqref{eqA}.
\end{prop}
\begin{proof}
	First, we show the existence of a positive solution.
	
	To this end, let $\psi^+_{\lambda}:H^1(\Omega)\rightarrow\RR$ be the $C^1$-functional defined by
	\begin{eqnarray*}
		\psi^+_{\lambda}(u)&=&\frac{1}{2}\gamma(u)+\frac{\mu}{2}||u^-||^2_2+\frac{c_{22}}{2}||u^+||^2_2+\frac{c_{23}}{2}||u^+||^{2^*}_{2^*}-\lambda c_{24}||u||^q\\
		&&\mbox{for some}\ c_{24}>0,\\
		&\geq&\frac{c_0}{2}||u||^2-\lambda c_{24}||u||^q\ (\mbox{recall that}\ c_{22}\geq\mu\ \mbox{and see (\ref{eq3})}),\\
		\Rightarrow\psi^+_{\lambda}(\cdot)&&\mbox{is coercive (recall that $q<2$)}.
	\end{eqnarray*}
	
	Also, $\psi^+_{\lambda}$ is sequentially weakly lower semicontinuous.
	
	So, we can find $\tilde{u}_{\lambda}\in H^1(\Omega)$ such that
	\begin{equation}\label{eq90}
		\psi^+_{\lambda}(\tilde{u}_{\lambda})=\inf[\psi^+_{\lambda}(u):u\in H^1(\Omega)].
	\end{equation}
	
	Since $q<2<2^*$, for $t\in(0,1)$ small we have
	\begin{eqnarray}\label{eq91}
		&&\psi^+_{\lambda}(t\tilde{u}_1)<0=\psi^+_{\lambda}(0),\nonumber\\
		&\Rightarrow&\psi^+_{\lambda}(\tilde{u}_{\lambda})<0=\psi^+_{\lambda}(0)\ (\mbox{see (\ref{eq90})}),\nonumber\\
		&\Rightarrow&\tilde{u}_{\lambda}\neq 0.
	\end{eqnarray}
	
	From (\ref{eq90}) we have
	\begin{eqnarray}\label{eq92}
		&&(\psi^+_{\lambda})'(\tilde{u}_{\lambda})=0,\nonumber\\
		&\Rightarrow&\left\langle A(\tilde{u}_{\lambda}),h\right\rangle+\int_{\Omega}\xi(z)\tilde{u}^+_{\lambda}hdz-
\int_{\Omega}(\xi(z)+\mu)\tilde{u}^-_{\lambda}hdz+\int_{\partial\Omega}\beta(z)\tilde{u}_{\lambda}hd\sigma\nonumber\\
		&&=\lambda c_{24}\int_{\Omega}(\tilde{u}^+_{\lambda})^{q-1}hdz-c_{23}\int_{\Omega}(\tilde{u}^+_{\lambda})^{2^*-1}hdz\ \mbox{for all}\ h\in H^1(\Omega).
	\end{eqnarray}
	
	In (\ref{eq92}) we choose $h=-\tilde{u}^-_{\lambda}\in H^1(\Omega)$. Then
	\begin{eqnarray*}
		&&\gamma(\tilde{u}^-_{\lambda})+\mu||\tilde{u}^-_{\lambda}||^2_2=0,\\
		&\Rightarrow&c_0||\tilde{u}^-_{\lambda}||^2\leq 0\ (\mbox{see (\ref{eq3})}),\\
		&\Rightarrow&\tilde{u}_{\lambda}\geq 0,\tilde{u}_{\lambda}\neq 0\ (\mbox{see (\ref{eq91})}).
	\end{eqnarray*}
	
	Therefore equation (\ref{eq92}) becomes
	\begin{eqnarray}\label{eq93}
		&&\left\langle A(\tilde{u}_{\lambda}),h\right\rangle+\int_{\Omega}\xi(z)\tilde{u}_{\lambda}hdz+\int_{\partial\Omega}\beta(z)\tilde{u}_{\lambda}hd\sigma\nonumber\\
		&&=\lambda\int_{\Omega}\tilde{u}^{q-1}_{\lambda}hdz-c_{22}\int_{\Omega}\tilde{u}_{\lambda}hdz-c_{23}\int_{\Omega}\tilde{u}^{2^*-1}_{\lambda}hdz\ \mbox{for all}\ h\in H^1(\Omega),\nonumber  \\
		&\Rightarrow&-\Delta\tilde{u}_{\lambda}(z)+\xi(z)\tilde{u}_{\lambda}(z)=\nonumber  \\
		&&\lambda\tilde{u}_{\lambda}(z)^{q-1}-c_{22}\tilde{u}_{\lambda}(z)-c_{23}\tilde{u}_{\lambda}(z)^{2^*-1}\ \mbox{for a.a.}\ z\in\Omega,\\
		&&\frac{\partial \tilde{u}_{\lambda}}{\partial n}+\beta(z)\tilde{u}_{\lambda}=0\ \mbox{on}\ \partial\Omega\ (\mbox{see Papageorgiou and R\u adulescu \cite{19}}),\nonumber\\
		&\Rightarrow&\tilde{u}_{\lambda}\in C_+\backslash\{0\}\nonumber
	\end{eqnarray}
	(as before using the regularity theory of Wang \cite{28}).
	
	From (\ref{eq93}) we have
	\begin{eqnarray*}
		&&\Delta\tilde{u}_{\lambda}(z)\leq[||\xi^+||_{\infty}+c_{23}||\tilde{u}_{\lambda}||^{2^*-2}_{\infty}+c_{22}]\tilde{u}_{\lambda}(z)\ \mbox{for almost all}\ z\in\Omega,\\
		&\Rightarrow&\tilde{u}_{\lambda}\in D_+\ \mbox{(by the strong maximum principle)}.
	\end{eqnarray*}
	
	Next we show the uniqueness of this positive solution. So, suppose that $\tilde{u}_{\lambda},\bar{u}_{\lambda}\in D_+$ are two solutions of \eqref{eqA}. The solution set of \eqref{eqA} is downward directed (see \cite{5}). So, we may assume that $\bar{u}_{\lambda}\leq\tilde{u}_{\lambda}$,
	\begin{eqnarray}
		&&\int_{\Omega}(D\tilde{u}_{\lambda},D\bar{u}_{\lambda})_{\RR^N}dz+\int_{\Omega}\xi(z)\tilde{u}_{\lambda}\bar{u}_{\lambda}dz+\int_{\partial\Omega}\beta(z)\tilde{u}_{\lambda}\bar{u}_{\lambda}d\sigma=\lambda\int_{\Omega}\tilde{u}^{q-1}_{\lambda}\bar{u}_{\lambda}dz-\nonumber\\
		&&\hspace{1cm}c_{22}\int_{\Omega}\tilde{u}_{\lambda}\bar{u}_{\lambda}dz-c_{23}\int_{\Omega}\tilde{u}^{2^*-1}_{\lambda}\bar{u}_{\lambda}dz\label{eq94}\\
		&&\int_{\Omega}(D\bar{u}_{\lambda},D\tilde{u}_{\lambda})_{\RR^N}dz+\int_{\Omega}\xi(z)\bar{u}_{\lambda}\tilde{u}_{\lambda}dz+\int_{\Omega}\beta(z)\bar{u}_{\lambda}\tilde{u}_{\lambda}d\sigma=\lambda\int_{\Omega}\bar{u}_{\lambda}^{q-1}\tilde{u}_{\lambda}dz-\nonumber\\
		&&\hspace{1cm}c_{22}\int_{\Omega}\bar{u}_{\lambda}\tilde{u}_{\lambda}dz-
c_{23}\int_{\Omega}\bar{u}^{2^*-1}_{\lambda}\tilde{u}_{\lambda}dz.\label{eq95}
	\end{eqnarray}
	
	We subtract (\ref{eq95}) from (\ref{eq94}) and obtain
	\begin{eqnarray*}
		&&\lambda\int_{\Omega}\tilde{u}_{\lambda}\bar{u}_{\lambda}\left[\frac{1}{\tilde{u}_{\lambda}^{2-q}}-\frac{1}{\bar{u}_{\lambda}^{2-q}}\right]=c_{23}\int_{\Omega}\tilde{u}_{\lambda}\bar{u}_{\lambda}\left(\tilde{u}_{\lambda}^{2^*-2}-\bar{u}^{2^*-2}_{\lambda}\right)dz\\
		&\Rightarrow&\tilde{u}_{\lambda}=\bar{u}_{\lambda}\ (\mbox{recall}\ q<2<2^*\ \mbox{and}\ \bar{u}_{\lambda}\leq\tilde{u}_{\lambda}).
	\end{eqnarray*}
	
	This proves the uniqueness of the positive solution $\tilde{u}_{\lambda}\in D_+$ of \eqref{eqA}.
	
	Problem \eqref{eqA} is odd. So, it follows that
	$$\tilde{u}_{\lambda}=-\tilde{u}_{\lambda}\in-D_+$$
	is the unique negative solution of \eqref{eqA}.
\end{proof}
\begin{remark}
	We present an alternative way of proving the uniqueness of the positive solution $\tilde{u}_{\lambda}\in D_+$ of \eqref{eqA}. So, let $\tilde{u}_{\lambda},\bar{u}_{\lambda}\in D_+$ be two positive solutions of \eqref{eqA}. Let $t>0$ be the biggest positive real such that
	\begin{equation}\label{eq96}
		t\bar{u}_{\lambda}\leq\tilde{u}_{\lambda}
	\end{equation}
	(see Filippakis and Papageorgiou \cite[Lemma 3.6]{5}). Suppose that $t\in(0,1)$. Since $q<2<2^*$, we can find $\hat{\vartheta}>0$ such that the function
	$$x\mapsto\lambda x^{q-1}+(\vartheta-c_{22})x-c_{23}x^{2^*-1}$$
	is increasing on $[0,\max\{||\tilde{u}_{\lambda}||_{\infty},||\bar{u}_{\lambda}||_{\infty}\}]$.

We have
\begin{eqnarray}\label{eq97}
	&&-\Delta\tilde{u}_{\lambda}(z)+(\xi(z)+\hat{\vartheta})\tilde{u}_{\lambda}(z)\nonumber\\
	&=&\lambda \tilde{u}_{\lambda}(z)^{q-1}+(\hat{\vartheta}-c_{22})\tilde{u}_{\lambda}(z)-c_{23}\tilde{u}_{\lambda}(z)^{2^*-1}\nonumber\\
	&\geq&\lambda(t\bar{u}_{\lambda}(z))^{q-1}+(\hat{\vartheta}-c_{22})(t\bar{u}_{\lambda}(z))-c_{23}(t\bar{u}_{\lambda}(z))^{2^*-1}\ (\mbox{see (\ref{eq96})})\nonumber\\
	&\geq&t\left[\lambda\bar{u}_{\lambda}(z)^{q-1}+(\hat{\vartheta}-c_{22})\bar{u}_{\lambda}(z)-c_{23}\bar{u}_{\lambda}(z)^{2^*-1}\right]\nonumber\\
	&&(\mbox{since}\ q<2<2^*\ \mbox{and}\ t\in(0,1))\nonumber\\
	&=&t\left[-\Delta\bar{u}_{\lambda}(z)+(\xi(z)+\hat{\vartheta})\bar{u}_{\lambda}(z)\right]\nonumber\\
	&=&-\Delta(t\bar{u}_{\lambda})(z)+(\xi(z)+\hat{\vartheta})(t\bar{u}_{\lambda})(z)\ \mbox{for almost all}\ z\in\Omega,\nonumber\\
	\Rightarrow&&\Delta(\tilde{u}_{\lambda}-t\bar{u}_{\lambda})(z)\leq(||\xi^+||_{\infty}+\hat{\vartheta})(\tilde{u}_{\lambda}-t\bar{u}_{\lambda})(z)\nonumber  \\
	&&
	\mbox{for almost all}\ z\in\Omega \ \ (\mbox{see hypothesis}\ H(\xi)').
\end{eqnarray}

Note that $\tilde{u}_{\lambda}\not\equiv t\bar{u}_{\lambda}$. Indeed, if $\tilde{u}_{\lambda}\equiv t\bar{u}_{\lambda}$, then
\begin{eqnarray*}
	&&\lambda(t\bar{u}_{\lambda})(z)^{q-1}-c_{22}(t\bar{u}_{\lambda})(z)-c_{23}(t\bar{u}_{\lambda})(z)^{2^*-1}\\
	&>&t\left[-\Delta\bar{u}_{\lambda}(z)+\xi(z)\bar{u}_{\lambda}(z)\right]\ (\mbox{since}\ q<2<2^*\ \mbox{and}\ t\in(0,1))\\
	&=&-\Delta\tilde{u}_{\lambda}(z)+\xi(z)\tilde{u}_{\lambda}(z)\ (\mbox{since}\ \tilde{u}_{\lambda}\equiv t\bar{u}_{\lambda})\\
	&=&\lambda\tilde{u}_{\lambda}(z)^{q-1}-c_{22}\tilde{u}_{\lambda}(z)-c_{23}\tilde{u}_{\lambda}(z)^{2^*-1}\\
	&=&\lambda(t\bar{u}_{\lambda})(z)^{q-1}-c_{22}(t\bar{u}_{\lambda})(z)-c_{23}(t\bar{u}_{\lambda})(z)^{2^*-1},
\end{eqnarray*}
a contradiction.

Then from (\ref{eq97}) and the strong maximum principle, we have
$$\tilde{u}_{\lambda}-t\bar{u}_{\lambda}\in D_+,$$
which contradicts the maximality of $t>0$. Hence $t\geq 1$ and we have
$$\bar{u}_{\lambda}\leq \tilde{u}_{\lambda}\ (\mbox{see (\ref{eq96})}).$$

Interchanging the roles of $\tilde{u}_{\lambda}$ and $\bar{u}_{\lambda}$ in the above argument, we also have
\begin{eqnarray*}
	&&\tilde{u}_{\lambda}\leq\bar{u}_{\lambda},\\
	&\Rightarrow&\tilde{u}_{\lambda}=\bar{u}_{\lambda}.
\end{eqnarray*}

This proves the uniqueness of the positive solution of \eqref{eqA}.
\end{remark}

Next we show that $\tilde{u}_{\lambda}\in D_+$ (respectively $\tilde{v}_{\lambda}\in -D_+$) is a lower bound (upper bound) for the elements of $S^+_{\lambda}$ (respectively $S^-_{\lambda}$).
\begin{prop}\label{prop17}
	If hypotheses $H(\xi)',H(\beta),H(f)_3$ hold, then
	\begin{itemize}
		\item[(a)] $\tilde{u}_{\lambda}\leq u$ for all $u\in S^+_{\lambda},\ \lambda\in(0,\lambda^+_*)$;
		\item[(b)] $v\leq\tilde{v}_{\lambda}$ for all $v\in S^-_{\lambda},\ \lambda\in(0,\lambda^-_*)$.
	\end{itemize}
\end{prop}
\begin{proof}
	\textit{(a)} Let $u\in S^+_{\lambda}$ and let $g^+_{\lambda}:\Omega\times\RR\rightarrow\RR$ be the Carath\'eodory function defined by
	\begin{eqnarray}\label{eq98}
		&&g^+_{\lambda}(z,x)=\left\{\begin{array}{ll}
			0&\mbox{if}\ x<0\\
			\lambda x^{q-1}+(\mu-c_{22})x-c_{23}x^{2^*-1}&\mbox{if}\ 0\leq x\leq u(z)\\
			\lambda u(z)^{q-1}+(\mu-c_{22})u(z)-c_{23}u(z)^{2^*-1}&\mbox{if}\ u(z)<x.
		\end{array}\right.
	\end{eqnarray}
	
	We set $G^+_{\lambda}(z,x)=\int^x_0g^+_{\lambda}(z,s)ds$ and consider the $C^1$-functional $\hat{\psi}^+_{\lambda}:H^1(\Omega)\rightarrow\RR$ defined by
	$$\hat{\psi}^+_{\lambda}(y)=\frac{1}{2}\gamma(y)+\frac{\mu}{2}||y||^2_2-\int_{\Omega}G^+_{\lambda}(z,y)dz\ \mbox{for all}\ y\in H^1(\Omega).$$
	
	From (\ref{eq3}) and (\ref{eq98}) it is clear that $\hat{\psi}^+_{\lambda}$ is coercive. Also, it is sequentially weakly lower semicontinuous. So, we can find $\bar{u}_{\lambda}\in H^1(\Omega)$ such that
	\begin{equation}\label{eq99}
		\hat{\psi}^+_{\lambda}(\bar{u}_{\lambda})=\inf[\hat{\psi}^+_{\lambda}(u):u\in H^1(\Omega)].
	\end{equation}
	
	As before, since $q<2<2^*$, we have
	\begin{eqnarray}\label{eq100}
		&&\hat{\psi}^+_{\lambda}(\bar{u}_{\lambda})<0=\hat{\psi}^+_{\lambda}(0),\nonumber\\
		&\Rightarrow&\bar{u}_{\lambda}\neq 0.
	\end{eqnarray}
	
	From (\ref{eq99}), we have
	\begin{eqnarray}\label{eq101}
		&&(\hat{\psi}^+_{\lambda})'(\bar{u}_{\lambda})=0,\nonumber\\
		&\Rightarrow&\left\langle A(\bar{u}_{\lambda}),h\right\rangle+\int_{\Omega}(\xi(z)+\mu)\bar{u}_{\lambda}hdz+\int_{\partial\Omega}\beta(z)\bar{u}_{\lambda}hd\sigma=\int_{\Omega}g^+_{\lambda}(z,\bar{u}_{\lambda})hdz\\
		&&\mbox{for all}\ h\in H^1(\Omega).\nonumber
	\end{eqnarray}
	
	In (\ref{eq101}) first we choose $h=-\bar{u}^-_{\lambda}\in H^1(\Omega)$. Then
	\begin{eqnarray*}
		&&\gamma(\bar{u}^-_{\lambda})+\mu||\bar{u}^-_{\lambda}||^2_2=0\ (\mbox{see (\ref{eq98})}),\\
		&\Rightarrow&c_0||\bar{u}^-_{\lambda}||^2\leq 0\ (\mbox{see (\ref{eq3})}),\\
		&\Rightarrow&\bar{u}_{\lambda}\geq 0,\ \bar{u}_{\lambda}\neq 0\ (\mbox{see (\ref{eq100})}).
	\end{eqnarray*}
	
	Also in (\ref{eq101}) we choose $h=(\bar{u}_{\lambda}-u)^+\in H^1(\Omega)$. Then
	\begin{eqnarray*}
		&&\left\langle A(\bar{u}_{\lambda}),(\bar{u}_{\lambda}-u)^+\right\rangle+\int_{\Omega}(\xi(z)+\mu)\bar{u}_{\lambda}(\bar{u}_{\lambda}-u)^+dz+\int_{\partial\Omega}\beta(z)\bar{u}_{\lambda}(\bar{u}_{\lambda}-u)^+d\sigma\\
		&=&\int_{\Omega}\left[\lambda u^{q-1}+(\mu-c_{22})u-c_{23}u^{2^*-1}\right](\bar{u}_{\lambda}-u)^+dz\ (\mbox{see (\ref{eq98})})\\
		&\leq&\int_{\Omega}[\lambda u^{q-1}+f(z,u)+\mu u](\bar{u}_{\lambda}-u)^+dz\ (\mbox{see (\ref{eq89})})\\
		&=&\left\langle A(u),(\bar{u}_{\lambda}-u)^+\right\rangle+\int_{\Omega}(\xi(z)+\mu)u(\bar{u}_{\lambda}-u)^+dz+\int_{\partial\Omega}\beta(z)u(\bar{u}_{\lambda}-u)^+d\sigma\\
		&&(\mbox{since}\ u\in S^+_{\lambda}),\\
		\Rightarrow&&\gamma((\bar{u}_{\lambda}-u)^+)+\mu||(\bar{u}_{\lambda}-u)^+||^2_2\leq 0,\\
		\Rightarrow&&c_0||(\bar{u}_{\lambda}-u)||^2\leq 0\ (\mbox{see (\ref{eq3})}),\\
		\Rightarrow&&\bar{u}_{\lambda}\leq u.
	\end{eqnarray*}
	
	So, we have proved that
	\begin{equation}\label{eq102}
		\bar{u}_{\lambda}\in[0,u],\ \bar{u}_{\lambda}\neq 0.
	\end{equation}
	
	On account of (\ref{eq98}) and (\ref{eq102}), equation (\ref{eq101}) becomes
	\begin{eqnarray*}
		&&\left\langle A(\bar{u}_{\lambda}),h\right\rangle+\int_{\Omega}\xi(z)\bar{u}_{\lambda}hdz+\int_{\partial\Omega}\beta(z)\bar{u}_{\lambda}hd\sigma=\int_{\Omega}\left[\lambda\bar{u}_{\lambda}^{q-1}-c_{22}\bar{u}_{\lambda}-c_{23}\bar{u}_{\lambda}^{2^*-1}\right]hdz\\
		&&\mbox{for all}\ h\in H^1(\Omega),\\
		&\Rightarrow&\bar{u}_{\lambda}\ \mbox{is a positive solutions of \eqref{eqA}},\\
		&\Rightarrow&\bar{u}_{\lambda}=\tilde{u}_{\lambda}\in D_+\ (\mbox{see Proposition \ref{prop16}}).
	\end{eqnarray*}
	
	From (\ref{eq102}) we conclude that
	$$\tilde{u}_{\lambda}\leq u\ \mbox{for all}\ u\in S^+_{\lambda}.$$
	
	\textit{(b)} In a similar fashion we show that
	$$v\leq\tilde{v}_{\lambda}\ \mbox{for all}\ v\in S^-_{\lambda}.$$
\end{proof}

Now we are ready to produce extremal constant sign solutions for problem \eqref{eqP}, that is, a smallest element for the set $S^+_{\lambda}$ $(\lambda\in (0,\lambda^+_*))$ and a biggest element for the set $S^-_{\lambda}$ $(\lambda\in(0,\lambda^-_*))$.
\begin{theorem}\label{th18}
	If hypotheses $H(\xi)',H(\beta),H(f)_3$ hold, then
	\begin{itemize}
		\item[(a)] for every $\lambda\in(0,\lambda^+_*)$ problem $S^+_{\lambda}$ has a smallest element $u^*_{\lambda}\in D_+$;
		\item[(b)] for every $\lambda\in(0,\lambda^-_*)$ problem $S^-_{\lambda}$ has a biggest element $v^*_{\lambda}\in-D_+.$
	\end{itemize}
\end{theorem}
\begin{proof}
	\textit{(a)} Recall that $S^+_{\lambda}$ $(\lambda\in(0,\lambda^+_*))$ is downward directed. So, invoking Lemma 3.10 of Hu and Papageorgiou \cite[p. 178]{9}, we can find a decreasing sequence $\{u_n\}_{n\geq 1}\subseteq S^+_{\lambda}$ such that
	$$\inf S^+_{\lambda}=\inf\limits_{n\geq 1}u_n.$$
	
	Evidently, $\{u_n\}_{n\geq 1}\subseteq H^1(\Omega)$ is bounded. So, we may assume that
	\begin{equation}\label{eq103}
		u_n\stackrel{w}{\rightarrow}u^*_{\lambda}\ \mbox{in}\ H^1(\Omega)\ \mbox{and}\ u_n\rightarrow u^*_{\lambda}\ \mbox{in}\ L^{\frac{2s}{s-1}}(\Omega)\ \mbox{and in}\ L^2(\partial\Omega).
	\end{equation}
	
	We have
	\begin{eqnarray}\label{eq104}
		&&\left\langle A(u_n),h\right\rangle+\int_{\Omega}\xi(z)u_nhdz+\int_{\partial\Omega}\beta(z)u_nhd\sigma=\lambda\int_{\Omega}u^{q-1}_nhdz+\int_{\Omega}f(z,u_n)hdz\\
		&&\mbox{for all}\ h\in H^1(\Omega),\ \mbox{all}\ n\in\NN.\nonumber
	\end{eqnarray}
	
	In (\ref{eq104}) we pass to the limit as $n\rightarrow\infty$ and use (\ref{eq103}). We obtain
	\begin{eqnarray}\label{eq105}
		&&\left\langle A(u^*_{\lambda}),h\right\rangle+\int_{\Omega}\xi(z)u^*_{\lambda}hdz+\int_{\partial\Omega}\beta(z)u^*_{\lambda}hd\sigma=\nonumber  \\
		&&
		\lambda\int_{\Omega}(u^*_{\lambda})^{q-1}hdz+\int_{\Omega}f(z,u^*_{\lambda})hdz\\
		&&\mbox{for all}\ h\in H^1(\Omega).\nonumber
	\end{eqnarray}
	
	Also we have
	\begin{eqnarray}\label{eq106}
		&&\tilde{u}_{\lambda}\leq u_n\ \mbox{for all}\ n\in\NN\ (\mbox{see Proposition \ref{prop17}}),\nonumber\\
		&\Rightarrow&\tilde{u}_{\lambda}\leq u^*_{\lambda}\ (\mbox{see (\ref{eq103})}).
	\end{eqnarray}
	
	Then from (\ref{eq105}) and (\ref{eq106}) we infer that
	$$u^*_{\lambda}\in S^+_{\lambda}\ \mbox{and}\ u^*_{\lambda}=\inf S^+_{\lambda}.$$
	
	\textit{(b)} Similarly we produce $v^*_{\lambda}\in-D_+$ the biggest element of $S^-_{\lambda}$.
\end{proof}

Using these extremal constant sign solutions of \eqref{eqP}, we can generate nodal (that is, sign changing) solutions.
\begin{prop}\label{prop19}
	If hypotheses $H(\xi)',H(\beta),H(f)_3$ hold and $\lambda\in(0,\lambda_*)$ (recall $\lambda_*=\min\{\lambda^+_{\lambda},\lambda^-_*\}$), then problem \eqref{eqP} admits a nodal solution $\hat{y}\in C^1(\overline{\Omega})\backslash\{0\}$.
\end{prop}
\begin{proof}
	Let $u^*_{\lambda}\in D_+$ and $v^*_{\lambda}\in-D_+$ be the two extremal constant sign solutions of \eqref{eqP} produced in Theorem \ref{th18}. Using them we introduce the following Carath\'eodory function
	{\tiny
	\begin{eqnarray}\label{eq107}
		k_{\lambda}(z,x)=\left\{\begin{array}{ll}
			\lambda|v^*_{\lambda}(z)|^{q-2}v^*_{\lambda}(z)+f(z,v^*_{\lambda}(z))+\mu v^*_{\lambda}(z)&\mbox{if}\ x<v^*_{\lambda}(z)\\
			\lambda|x|^{q-2}x+f(z,x)+\mu x&\mbox{if}\ v^*_{\lambda}(z)\leq x\leq u^*_{\lambda}(z)\\
			\lambda u^*_{\lambda}(z)^{q-1}+f(z,u^*_{\lambda}(z))+\mu u^*_{\lambda}(z)&\mbox{if}\ u^*_{\lambda}(z)<x.
		\end{array}\right.
	\end{eqnarray}
	}
	
	Let $K_{\lambda}(z,x)=\int^x_0k_{\lambda}(z,s)ds$ and consider the $C^1$-functional $\eta_{\lambda}:H^1(\Omega)\rightarrow\RR$ defined by
	$$\eta_{\lambda}(u)=\frac{1}{2}\gamma(u)+\frac{\mu}{2}||u||^2_2-\int_{\Omega}K_{\lambda}(z,u)dz\ \mbox{for all}\ u\in H^1(\Omega).$$
	
	Also we consider the positive and negative truncations of $k_{\lambda}(z,\cdot)$ (that is, $k^{\pm}_{\lambda}(z,x)=k_{\lambda}(z,\pm x^{\pm})$), set $K^{\pm}_{\lambda}(z,x)=\int^x_0k^{\pm}_{\lambda}(z,s)ds$ and consider the corresponding $C^1$-functionals $\eta^{\pm}_{\lambda}:H^1(\Omega)\rightarrow\RR$ defined by
	$$\eta^{\pm}_{\lambda}(u)=\frac{1}{2}\gamma(u)+\frac{\mu}{2}||u||^2_2-\int_{\Omega}K^{\pm}_{\lambda}(z,u)dz\ \mbox{for all}\ u\in H^1(\Omega).$$
	
	\begin{claim}\label{cl1}
		$K_{\eta_{\lambda}}\subseteq[v^*_{\lambda},u^*_{\lambda}]\cap C^1(\overline{\Omega}),\ K_{\eta^+_{\lambda}}=\{0,u^*_{\lambda}\},\ K_{\eta^-_{\lambda}}=\{0,v^*_{\lambda}\}$.
	\end{claim}
	
	Let $u\in K_{\eta_{\lambda}}$. We have
	\begin{equation}\label{eq108}
		\left\langle A(u),h\right\rangle+\int_{\Omega}(\xi(z)+\mu)uhdz+\int_{\partial\Omega}\beta(z)uhd\sigma=\int_{\Omega}k_{\lambda}(z,u)hdz\ \mbox{for all}\ h\in H^1(\Omega).
	\end{equation}
	
	In (\ref{eq108}) we choose $h=(u-u^*_{\lambda})^+\in H^1(\Omega)$. Then
	\begin{eqnarray*}
		&&\left\langle A(u),(u-u^*_{\lambda})^+\right\rangle+\int_{\Omega}(\xi(z)+\mu)u(u-u^*_{\lambda})^+dz+\int_{\partial\Omega}\beta(z)u(u-u^*_{\lambda})^+d\sigma\\
		&=&\int_{\Omega}[\lambda(u^*_{\lambda})^{q-1}+f(z,u^*_{\lambda})+\mu u^*_{\lambda}](u-u^*_{\lambda})^+dz\ (\mbox{see (\ref{eq107})})\\
		&=&\left\langle A(u^*_{\lambda}),(u-u^*_{\lambda})^+\right\rangle+\int_{\Omega}(\xi(z)+\mu)u^*_{\lambda}(u-u^*_{\lambda})^+dz+\int_{\partial\Omega}\beta(z)u^*_{\lambda}(u-u^*_{\lambda})^+d\sigma\\
		&&(\mbox{since}\ u^*_{\lambda}\in S^+_{\lambda}),\\
		&\Rightarrow&u\leq u^*_{\lambda}.
	\end{eqnarray*}
	
	Similarly, choosing $h=(v^*_{\lambda}-u)^+\in H^1(\Omega)$ in (\ref{eq108}), we show that
	$$v^*_{\lambda}\leq u.$$
	
	So, from the above and the regularity theory of Wang \cite{28}, we have
	\begin{eqnarray*}
		&&u\in[v^*_{\lambda},u^*_{\lambda}]\cap C^1(\overline{\Omega}),\\
		&\Rightarrow&K_{\eta_{\lambda}}\subseteq[v^*_{\lambda},u^*_{\lambda}]\cap C^1(\overline{\Omega}).
	\end{eqnarray*}
	
	In a similar fashion, we show that
	$$K_{\eta^+_{\lambda}}\subseteq[0,u^*_{\lambda}]\ \mbox{and}\ K_{\eta^-_{\lambda}}\subseteq[v^*_{\lambda},0].$$
	
	The extremality of $u^*_{\lambda}\in D_+$ and of $v^*_{\lambda}\in-D_+$, implies that
	$$K_{\eta^+_{\lambda}}=\{0,u^*_{\lambda}\}\ \mbox{and}\ K_{\eta^-_{\lambda}}=\{0,v^*_{\lambda}\}.$$
	
	This proves Claim \ref{cl1}.
	\begin{claim}\label{cl2}
		$u^*_{\lambda}\in D_+$ and $v^*_{\lambda}\in -D_+$ are local minimizers of $\eta_{\lambda}$.
	\end{claim}
	
	Evidently, $\eta^+_{\lambda}$ is coercive (see (\ref{eq107})) and sequentially weakly lower semicontinuous. So, we can find $\tilde{u}^*_{\lambda}\in H^1(\Omega)$ such that
	\begin{equation}\label{eq109}
		\eta^+_{\lambda}(\tilde{u}^*_{\lambda})=\inf[\eta^+_{\lambda}(u):u\in H^1(\Omega)].
	\end{equation}
	
	As before, since $q<2<2^*$, we have
	\begin{eqnarray}\label{eq110}
		&&\eta^+_{\lambda}(\tilde{u}^*_{\lambda})<0=\eta^+_{\lambda}(0)\nonumber\\
		&\Rightarrow&\tilde{u}^*_{\lambda}\neq 0.
	\end{eqnarray}
	
	From (\ref{eq109}) we have that
	\begin{equation}\label{eq111}
		\tilde{u}^*_{\lambda}\in K_{\eta^+_{\lambda}}.
	\end{equation}
	
	From (\ref{eq110}), (\ref{eq111}) and Claim \ref{cl1} it follows that $\tilde{u}^*_{\lambda}=u^*_{\lambda}\in D_+$. Note that
	\begin{eqnarray*}
		&&\left.\eta_{\lambda}\right|_{C_+}=\left.\eta^+_{\lambda}\right|_{C_+}\ (\mbox{see (\ref{eq107})}),\\
		&\Rightarrow&u^*_{\lambda}\ \mbox{is a local}\ C^1(\overline{\Omega})-\mbox{minimizer of}\ \eta_{\lambda},\\
		&\Rightarrow&u^*_{\lambda}\ \mbox{is a local}\ H^1(\overline{\Omega})-\mbox{minimizer of}\ \eta_{\lambda}\ (\mbox{see Proposition \ref{prop4}}).
	\end{eqnarray*}
	
	Similarly for $v^*_{\lambda}\in-D_+$, using this time the functional $\eta^-_{\lambda}$.
		This proves Claim \ref{cl2}.
	
	Without any loss of generality, we may assume that
	$$\eta_{\lambda}(v^*_{\lambda})\leq\eta_{\lambda}(u^*_{\lambda}).$$
	
	The reasoning is similar if the opposite inequality holds. Also, we assume that $K_{\eta_{\lambda}}$ if finite or on account of Claim \ref{cl1} we already have an infinity of smooth nodal solutions and so we are done. Then Claim \ref{cl2} implies that we can find $\rho\in(0,1)$ small such that
	\begin{equation}\label{eq112}
		\eta_{\lambda}(v^*_{\lambda})\leq\eta_{\lambda}(u^*_{\lambda})<\inf[\eta_{\lambda}(u):||u-u^*_{\lambda}||=\rho]=m_{\lambda},
\ ||v^*_{\lambda}-u^*_{\lambda}||>\rho
	\end{equation}
	(see Aizicovici, Papageorgiou and Staicu \cite{1}, proof of Proposition 29).
	
	From (\ref{eq3}) and (\ref{eq107}) it is clear that $\eta_{\lambda}$ is coercive. Hence
	\begin{equation}\label{eq113}
		\eta_{\lambda}\ \mbox{satisfies the C-condition}.
	\end{equation}
	
	Then (\ref{eq112}) and (\ref{eq113}) permit the use of Theorem \ref{th1} (the mountain pass theorem). So, we can find $\hat{y}\in H^1(\Omega)$ such that
	\begin{equation}\label{eq114}
		\hat{y}\in K_{\eta_{\lambda}}\ \mbox{and}\ m_{\lambda}\leq\eta_{\lambda}(\hat{y}).
	\end{equation}
	
	Claim \ref{cl1} together with (\ref{eq112}) and (\ref{eq114}) imply that
	$$\hat{y}\in[v^*_{\lambda},u^*_{\lambda}]\cap C^1(\overline{\Omega}),\ \hat{y}\not\in\{u^*_{\lambda},v^*_{\lambda}\}.$$
	
	Since $\hat{y}$ is a critical point of $\eta_{\lambda}$ of mountain pass type, we have
	\begin{equation}\label{eq115}
		C_1(\eta_{\lambda},\hat{y})\neq 0
	\end{equation}
	(see Motreanu, Motreanu and Papageorgiou \cite[Corollary 6.81, p. 168]{15}).
	
	On the other hand, the presence of the concave term $\lambda|x|^{q-2}x\ (q<2)$ in the reaction function, hypothesis $H(f)_3(iii)$ and Lemma 3.4 of D'Agui, Marano and Papageorgiou \cite{3} imply that
	\begin{equation}\label{eq116}
		C_k(\eta_{\lambda},0)=0\ \mbox{for all}\ k\in\NN_0.
	\end{equation}
	
	Comparing (\ref{eq115}) and (\ref{eq116}), we conclude that
	\begin{eqnarray*}
		&&\hat{y}\neq 0,\\
		&\Rightarrow&\hat{y}\in C^1(\overline{\Omega})\backslash\{0\}\ \mbox{is nodal}.
	\end{eqnarray*}
\end{proof}

So, we can state the following multiplicity theorem for problem \eqref{eqP}.
\begin{theorem}\label{th20}
	If hypotheses $H(\xi)',H(\beta),H(f)_3$ hold, then we can find a parameter value $\lambda_*>0$ such that for every $\lambda\in(0,\lambda_*)$ problem \eqref{eqP} has at least five nontrivial smooth solutions
	\begin{eqnarray*}
		&&u_0,\hat{u}\in D_+,\ v_0,\hat{v}\in-D_+,\\
		&&\hat{y}\in C^1(\overline{\Omega})\ \mbox{nodal}.
	\end{eqnarray*}
\end{theorem}

If we strengthen the regularity of $f(z,\cdot)$, we can improve Theorem \ref{th20} and produce a sixth nontrivial smooth solution. However, we do not provide any sign information for this sixth solution.

The new conditions on the perturbation term $f(z,x)$ are the following:

\smallskip
$H(f)_4:$ $f:\Omega\times\RR\rightarrow\RR$ is a measurable function such that for almost all $z\in\Omega,\ f(z,\cdot)\in C^1(\RR)$ and
\begin{itemize}
	\item[(i)] for every $\rho>0$, there exists $a_{\rho}\in L^{\infty}(\Omega)$ such that
	\begin{eqnarray*}
		&&|f'_x(z,x)|\leq a_{\rho}(z)\ \mbox{for almost all}\ z\in\Omega,\ \mbox{all}\ |x|\leq\rho\\
		&\mbox{and}&\lim\limits_{x\rightarrow\pm\infty}\frac{f'_x(z,x)}{|x|^{2^*-2}}=0\ \mbox{uniformly for almost all}\ z\in\Omega;
	\end{eqnarray*}
	\item[(ii)] $\lim\limits_{x\rightarrow\pm\infty}\frac{F(z,x)}{x^2}=+\infty$ uniformly for almost all $z\in\Omega$ and there exists $e\in L^1(\Omega)$ such that
	$$\tau(z,x)\leq\tau(z,y)+e(z)\ \mbox{for almost all}\ z\in\Omega,\ \mbox{all}\ 0\leq x\leq y\ \mbox{and all}\ y\leq x\leq 0$$
	(recall that $F(z,x)=\int^x_0f(z,s)ds$ and $\tau(z,x)=f(z,x)x-2F(z,x)$);
	\item[(iii)] $f(z,0)=0$ for almost all $z\in\Omega$, $f'_x(z,0)=\lim\limits_{x\rightarrow 0}\frac{f(z,x)}{x}$ uniformly for almost all $z\in\Omega$ and
	$$f'_x(\cdot,0)\in L^{\infty}(\Omega),f'_x(z,0)\leq\hat{\lambda}_1\ \mbox{for almost all}\ z\in\Omega,f'_x(\cdot,0)\not\equiv\hat{\lambda}_1.$$
\end{itemize}

Under the above hypotheses, we have $\varphi_{\lambda}\in C^2(H^1(\Omega)\backslash\{0\})$.
\begin{prop}\label{prop21}
	If hypotheses $H(\xi)',H(\beta),H(f)_4$ hold and $\lambda\in(0,\lambda_*)$, then problem \eqref{eqP} admits a sixth nontrivial smooth solution
	$$\tilde{y}\in C^1(\overline{\Omega}).$$
\end{prop}
\begin{proof}
	From Proposition \ref{prop15} we know that
	\begin{equation}\label{eq117}
		C_k(\varphi_{\lambda},u_0)=C_k(\varphi_{\lambda},v_0)=\delta_{k,0}\ZZ\ \mbox{for all}\ k\in\NN_0.
	\end{equation}
	
	Also, recall (see the proof of Proposition \ref{prop15}), that $\hat{u}\in D_+$ is a critical point of mountain pass type for $\hat{\varphi}^+_{\lambda}$ and $\hat{v}\in-D_+$ is a critical point of mountain pass type for $\hat{\varphi}^-_{\lambda}$. Note that
	\begin{eqnarray}\label{eq118}
		&&\left.\varphi_{\lambda}\right|_{C_+}=\left.\hat{\varphi}^+_{\lambda}\right|_{C_+}\ \mbox{and}\ \left.\varphi_{\lambda}\right|_{-C_+}=\left.\hat{\varphi}^-_{\lambda}\right|_{-C_+}\ (\mbox{see (\ref{eq72}), (\ref{eq73})}),\nonumber\\
		&\Rightarrow&C_k(\left.\varphi_{\lambda}\right|_{C^1(\overline{\Omega})},\hat{u})=C_k(\left.\hat{\varphi}^+_{\lambda}\right|_{C^1(\overline{\Omega})},\hat{u})\ \mbox{and}\ C_k(\left.\varphi_{\lambda}\right|_{C^1(\overline{\Omega})},\hat{v})=C_k(\left.\hat{\varphi}^-_{\lambda}\right|_{C^1(\overline{\Omega})},\hat{v})\nonumber\\
		&&\mbox{for all}\ k\in\NN_0\ (\mbox{recall that}\ \hat{u}\in D_+\ \mbox{and}\ \hat{v}\in-D_+)\\
		&\Rightarrow&C_k(\varphi_{\lambda},\hat{u})=C_k(\hat{\varphi}^+_{\lambda},\hat{u})\ \mbox{and}\ C_k(\varphi_{\lambda},\hat{v})=\nonumber  \\
		&&
		C_k(\hat{\varphi}^-_{\lambda},\hat{v})\ \mbox{for all}\ k\in\NN_0\ (\mbox{see Palais \cite{17}})\nonumber\\
		&\Rightarrow& C_k(\varphi_{\lambda},\hat{u})=C_k(\varphi_{\lambda},\hat{v})=\delta_{k,1}\ZZ\nonumber\\
		&&
		\ \ (\mbox{see \cite[Corollary 6.102, p. 177]{15}})\nonumber.
	\end{eqnarray}
	
	Also, as we already pointed out in the proof of Proposition \ref{prop19} (see (\ref{eq116}) and recall $\left.\eta_{\lambda}\right|_{[v^*_{\lambda},u^*_{\lambda}]}=\left.\varphi_{\lambda}\right|_{[v^*_{\lambda},u^*_{\lambda}]}$), we have
	\begin{equation}\label{eq119}
		C_k(\varphi_{\lambda},0)=0\ \mbox{for all}\ k\in\NN_0.
	\end{equation}
	
	From Proposition \ref{prop6}, we have
	\begin{equation}\label{eq120}
		C_k(\varphi_{\lambda},\infty)=0\ \mbox{for all}\ k\in\NN_0.
	\end{equation}
	
	Let $\hat{m}=\max\{||u_0||_{\infty},||\hat{u}||_{\infty},||v_0||_{\infty},||\hat{v}||_{\infty}\}$. Hypotheses $H(f)_4(i),(iii)$ imply that we can find $\hat{\xi}>0$ such that for almost all $z\in\Omega$, the function
	$$x\mapsto f(z,x)+\hat{\xi}x$$
	is nondecreasing on $[-\hat{m},\hat{m}]$. With $\hat{y}\in C^1(\overline{\Omega})\backslash\{0\}$ being the nodal solution we have
	\begin{eqnarray*}
		&&-\Delta\hat{y}(z)+(\xi(z)+\hat{\xi})\hat{y}(z)\\
		&=&\lambda|\hat{y}(z)|^{q-2}\hat{y}(z)+f(z,\hat{y}(z))+\hat{\xi}\hat{y}(z)\\
		&\leq&\lambda u^*_{\lambda}(z)^{q-1}+f(z,u^*_{\lambda}(z))+\hat{\xi}u^*_{\lambda}(z)\ (\mbox{since}\ \hat{y}\leq u^*_{\lambda},\ \mbox{see Proposition \ref{prop19}})\\
		&=&-\Delta u^*_{\lambda}(z)+(\xi(z)+\hat{\xi})u^*_{\lambda}(z)\ \mbox{for almost all}\ z\in\Omega,\\
		\Rightarrow&&\Delta(u^*_{\lambda}-\hat{y})(z)\leq[||\xi^+||_{\infty}+\hat{\xi}](u^*_{\lambda}-\hat{y})(z)\ \mbox{for almost all}\ z\in\Omega\ (\mbox{see hypothesis}\ H(\xi))\\
		\Rightarrow&&u^*_{\lambda}-\hat{y}\in D_+\ \mbox{(by the strong maximum principle)}.
	\end{eqnarray*}
	
	Similarly we show that
	$$\hat{y}-v^*_{\lambda}\in D_+.$$
	
	So, finally we have
	\begin{equation}\label{eq121}
		\hat{y}\in {\rm int}_{C^1(\overline{\Omega})}[v^*_{\lambda},u^*_{\lambda}].
	\end{equation}
	
	Recall that
	\begin{eqnarray}\label{eq122}
		&&\eta_{\lambda}\left|_{[v^*_{\lambda},u^*_{\lambda}]}=\varphi_{\lambda}\right|_{[v^*_{\lambda},u^*_{\lambda}]}\ \mbox{see (\ref{eq107})},\nonumber\\
		&\Rightarrow&C_k(\eta_{\lambda},\hat{y})=C_k(\varphi_{\lambda},\hat{y})\ \mbox{for all}\ k\in\NN_0\nonumber\\
		&&(\mbox{as before from (\ref{eq121}) and Palais \cite{17}})\nonumber\\
		&\Rightarrow&C_k(\varphi_{\lambda},\hat{y})=\delta_{k,1}\ZZ\ \mbox{for all}\ k\in\NN_0\\
		&&(\mbox{since}\ \hat{y}\in K_{\eta_{\lambda}}\ \mbox{is of mountain pass type, see  \cite[p. 177]{15}}).\nonumber
	\end{eqnarray}
	
	Suppose that $K_{\varphi_{\lambda}}=\{0,u_0,v_0,\hat{u},\hat{v},\hat{y}\}$. Then from (\ref{eq117}), (\ref{eq118}), (\ref{eq119}), (\ref{eq120}) and (\ref{eq121}) and using the Morse relation with $t=-1$ (see (\ref{eq6})), we have
	$$2(-1)^0+2(-1)^1+(-1)^1=0,$$
	a contradiction. So, there exists $\tilde{y}\in H^1(\Omega)$ such that
	$$\tilde{y}\in K_{\varphi_{\lambda}}\subseteq C^1(\overline{\Omega})\ \mbox{and}\ \tilde{y}\notin\{0,u_0,v_0,\hat{u},\hat{v},\hat{y}\}.$$
	
	This is the sixth nontrivial smooth solution of problem \eqref{eqP}.
\end{proof}

So, we can state the following new multiplicity theorem for problem \eqref{eqP}.
\begin{theorem}\label{th22}
	If hypotheses $H(\xi)',H(\beta),H(f)_4$ hold, then there exists a parameter value $\lambda_*>0$ such that for every $\lambda\in(0,\lambda_*)$ problem \eqref{eqP} has at least six nontrivial smooth solutions
	\begin{eqnarray*}
		&&u_0,\hat{u}\in D_+,\ v_0,\hat{v}\in-D_+,\\
		&&\hat{y}\in C^1(\overline{\Omega})\ \mbox{nodal and}\ \tilde{y}\in C^1(\overline{\Omega}).
	\end{eqnarray*}
\end{theorem}

\section{Infinitely Many Solutions}

In this section, we generate an infinity of nontrivial smooth solutions by introducing symmetry on the reaction term. We prove two such results. The first concerns problem \eqref{eqP} and the solutions we produce are nodal. The second result deals with problem (\ref{eq1}) and produces an infinity of nontrivial smooth solutions but without any sign information.

For the first theorem, the hypotheses on the perturbation term $f(z,x)$ are the following:

\smallskip
$H(f)_5:$ $f:\Omega\times\RR\rightarrow\RR$ is a Carath\'eodory function which satisfies hypotheses $H(f)_3$ and in addition for almost all $z\in\Omega,\ f(z,\cdot)$ is odd.
\begin{theorem}\label{th23}
	If hypotheses $H(\xi)',H(\beta),H(f)_5$ hold, then we can find a parameter value $\lambda_*>0$ such that for every $\lambda\in(0,\lambda_*)$ problem \eqref{eqP} admits a sequence $\{u_n\}_{n\geq 1}\subseteq C^1(\overline{\Omega})$ of distinct nodal solutions such that
	$$u_n\rightarrow 0\ \mbox{in}\ C^1(\overline{\Omega})\ \mbox{as}\ n\rightarrow\infty.$$
\end{theorem}
\begin{proof}
	Let $\lambda_*=\min\{\lambda^+_*,\lambda^-_*\}$ be as in Proposition \ref{prop15}(c) and consider the $C^1$-functional $\eta_{\lambda}:H^1(\Omega)\rightarrow\RR$ introduced in the proof of Proposition \ref{prop19}. We have the following properties:
	\begin{itemize}
		\item	$\eta_{\lambda}$ is even;
		\item	$\eta_{\lambda}$ is coercive (hence it is bounded below and it satisfies the C-condition);
		\item $\eta_{\lambda}(0)=0$.
	\end{itemize}
	
	Let $V$ be a finite dimensional subspace of $H^1(\Omega)$. So, all norms on $V$ are equivalent. Also, from (\ref{eq107}) and hypotheses $H(f)_5(i),(iii)=H(f)_3(i),(iii)$ we see that we can find $c_{24}>0$ such that
	\begin{equation}\label{eq123}
		|K_{\lambda}(z,x)|\leq c_{24}|x|^2\ \mbox{for almost all}\ z\in\Omega.
	\end{equation}
	
	So, for every $u\in V$ and recalling that all norms are equivalent, we have
	\begin{eqnarray*}
		&&\eta_{\lambda}(u)\leq c_{25}||u||^2-\lambda c_{26}||u||^q\ \mbox{for some}\ c_{25},c_{26}>0\\
		&&(\mbox{see (\ref{eq123}) and hypotheses}\ H(\xi)',H(\beta)).
	\end{eqnarray*}
	
	Since $q<2$, we can find $\rho_{\lambda}\in(0,1)$ small such that
	$$\eta_{\lambda}(u)<0\ \mbox{for all}\ u\in V\ \mbox{with}\ ||u||=\rho_{\lambda}.$$
	
	Therefore we can apply Theorem \ref{th3} and find $\{u_n\}_{n\geq 1}\subseteq K_{\eta_\lambda}\subseteq[v^*_{\lambda},u^*_{\lambda}]\cap C^1(\overline{\Omega})$ (see Claim \ref{cl1} in the proof of Proposition \ref{prop19}) such that
	\begin{equation}\label{eq124}
		u_n\rightarrow 0\ \mbox{in}\ H^1(\Omega).
	\end{equation}
	
	From the regularity theory of Wang \cite{28} we know that
	\begin{eqnarray}\label{eq125}
		&&u_n\in C^{1,\alpha}(\overline{\Omega})\ \mbox{with}\ \alpha=1-\frac{N}{s}>0\ \mbox{and}\ ||u_n||_{C^{1,\alpha}(\overline{\Omega})}\leq c_{27}\\
		&&\mbox{for all}\ n\in\NN\ \mbox{and some}\ c_{27}>0.\nonumber
	\end{eqnarray}
	
	Exploiting the compact embedding of $C^{1,\alpha}(\overline{\Omega})$ into $C^1(\overline{\Omega})$, from (\ref{eq124}) and (\ref{eq125}) we infer that
	$$u_n\rightarrow 0\ \mbox{in}\ C^1(\overline{\Omega})\ \mbox{as}\ n\rightarrow\infty .$$
	
	Moreover, since $u_n\in[v^*_{\lambda},u^*_{\lambda}]$ for all $n\in\NN$, $u_n\in C^1(\overline{\Omega})$ is nodal.
\end{proof}

The second result of this section is about problem (\ref{eq1}). For this result the hypotheses on the reaction term $f(z,x)$ are the following:

\smallskip
$H(f)_6:$ $f:\Omega\times\RR\rightarrow\RR$ is a Carath\'eodory function such that for almost all $z\in\Omega$ $f(z,x)$ is odd and hypotheses $H(f)_6(i),(ii)$ are the same as the corresponding hypotheses $H(f)_5(i),(ii)$.
\begin{remark}
	We point out that in the above hypotheses there are no conditions on $f(z,\cdot)$ near zero.
\end{remark}
\begin{theorem}\label{th24}
	If hypotheses $H(\xi),H(\beta),H(f)_6$ hold, then problem (\ref{eq1}) admits an unbounded sequence of solutions $\{u_n\}_{n\geq 1}\subseteq C^1(\overline{\Omega})$.
\end{theorem}
\begin{proof}
	From Proposition \ref{prop5} we know that the energy (Euler) functional $\varphi$ satisfies the C-condition and $\varphi(0)=0$.
	
	We consider the following orthogonal direct sum decomposition of $H^1(\Omega)$
	$$H^1(\Omega)=H_-\oplus E(0)\oplus H_+$$
	with $H_-=\overset{m_-}{\underset{\mathrm{k=1}}\oplus}E(\hat{\lambda}_k),\ H_+=\overline{{\underset{\mathrm{k\geq m_+}}\oplus}E(\hat{\lambda}_k)}$ (see Section 2). Then every $u\in H^1(\Omega)$ admits a unique sum decomposition of the term
	$$u=\bar{u}+u^0+\hat{u}\ \mbox{with}\ \bar{u}\in H_-,\ u^{\circ}\in E(0),\ \hat{u}\in H_+\,.$$
	
	Hypothesis $H(f)_6(i)$ implies that given $\epsilon>0$, we can find $c_{28}=c_{28}(\epsilon)>0$ such that
	\begin{equation}\label{eq126}
		F(z,x)\leq\frac{\epsilon}{2}|x|^{2^*}+c_{28}|x|\ \mbox{for almost all}\ z\in\Omega,\ \mbox{all}\ x\in\RR.
	\end{equation}
	
	Let $u\in H_+$. We have
	\begin{eqnarray}\label{eq127}
		\varphi(u)&=&\frac{1}{2}\gamma(u)-\int_{\Omega}F(z,u)dz\nonumber\\
		&\geq&\frac{1}{2}\gamma(u)-\frac{\epsilon}{2}||u||^{2^*}_{2^*}-c_{28}||u||_1\ (\mbox{see (\ref{eq128})})\nonumber\\
		&\geq&\frac{1}{2}\left[\gamma(u)-\epsilon c_{29}||u||^{2^*}\right]-c_{28}||u||_1\ \mbox{for some}\ c_{29}>0\nonumber\\
		&\geq&\left[2c_{30}||u||^2-\epsilon c_{29}||u||^{2^*}\right]-c_{28}||u||_1\ \mbox{for some}\ c_{29}>0\ (\mbox{recall that}\ u\in H_+)\nonumber\\
		&=&\left[c_{30}||u||^2-\epsilon c_{29}||u||^{2^*}\right]+\left[c_{30}||u||^2-c_{28}||u||_1\right].
	\end{eqnarray}
	
	If $\hat{u}\in H_+$ is such that
	$$||\hat{u}||=\hat{\rho}<\left(\frac{c_{30}}{\epsilon c_{29}}\right)^{\frac{1}{2^*-2}},$$
	then we have
	$$c_{30}||u||^2-\epsilon c_{29}||u||^{2^*}>0.$$
	
	Also, for $l\geq m_+$ and $u\in\overline{{\underset{\mathrm{k\geq l}}\oplus}E(\hat{\lambda}_k)}$, we have
	$$c_{28}||u||_1\leq c_{31}||u||_2\leq\frac{c_{31}}{\sqrt{\hat{\lambda}_l}}||u||\ \mbox{for some}\ c_{31}>0.$$
	
	Therefore
	$$c_{30}||u||^2-c_{28}||u||_1\geq c_{30}||u||^2-\frac{c_{31}}{\sqrt{\hat{\lambda}_l}}||u||\ \mbox{for}\ u\in V_l=\overline{{\underset{\mathrm{k\geq l}}\oplus}E(\hat{\lambda}_k)}.$$
	
	So, if $u\in V_l$ with $l\geq m_+$ big and with $||u||=\hat{\rho}$, we have
	$$c_{30}\hat{\rho}^2-\frac{c_{31}}{\sqrt{\hat{\lambda}_l}}\hat{\rho}>0.$$
	
	Returning to (\ref{eq127}), we see that
	\begin{equation}\label{eq128}
		\left.\varphi\right|_{V_l\cap\partial B_{\hat{\rho}}}>0.
	\end{equation}
	
	Next let $Z\subseteq H^1(\Omega)$ be a finite dimensional subspace. From hypotheses $H(f)_6(i),(ii)$, we know that given any $\eta>0$, we can find $c_{32}=c_{32}(\eta)>0$ such that
	\begin{equation}\label{eq129}
		F(z,x)\geq\frac{\eta}{2}x^2-c_{32}\ \mbox{for almost all}\ z\in\Omega,\ \mbox{all}\ x\in\RR.
	\end{equation}
	
	For $u\in Z$ we have
	\begin{eqnarray}\label{eq130}
		\varphi(u)&=&\frac{1}{2}\gamma(u)-\int_{\Omega}F(z,u)dz\nonumber\\
		&\leq&\frac{1}{2}\gamma(u)-\frac{\eta}{2}||u||^2_2+c_{32}|\Omega|_N\ \mbox{see (\ref{eq129})}\nonumber\\
		&\leq&c_{33}||u||^2-\eta c_{34}||u||^2+c_{32}|\Omega|_N\ \mbox{for some}\ c_{33},c_{34}>0
	\end{eqnarray}
	(since $Z$ is finite dimensional all norms are equivalent).
	
	But $\eta>0$ is arbitrary. So, we choose $\eta>\frac{c_{33}}{c_{34}}>0$ and from (\ref{eq130}) we infer that
	\begin{equation}\label{eq131}
		\varphi|_{Z}\ \mbox{is anticoercive}.
	\end{equation}
	
	Then (\ref{eq128}) and (\ref{eq131}) permit the use of Theorem \ref{th2} (the symmetric mountain pass theorem). So, we can find $\{u_n\}_{n\geq 1}\subseteq H^1(\Omega)$ such that
	$$u_n\in K_{\varphi}\ \mbox{for all}\ n\in\NN\ \mbox{and}\ ||u_n||\rightarrow+\infty\,.$$
	
	Hence $u_n$ is a solution of (\ref{eq1}) and $u_n\in C^1(\overline{\Omega})$ with $||u_n||_{C^1(\overline{\Omega})}\rightarrow+\infty$.
\end{proof}

\medskip{\bf Acknowledgments.} 	This research was supported by the Slovenian Research Agency grants P1-0292, J1-8131, J1-7025.

\end{document}